\documentclass[a4paper]{article}

\usepackage[plainpages=false,colorlinks=true,
            linkcolor=black,urlcolor=black,citecolor=black]{hyperref}
\usepackage{geometry}

\usepackage{amsmath,amssymb,amsthm, amsfonts,mathrsfs,cite}
\usepackage{color}
\usepackage{booktabs}
\usepackage{tikz}
\usepackage{scrextend}
\usepackage{arydshln}
\usepackage[font=scriptsize,labelfont=bf]{caption}

\usepackage{algpseudocode}
\usepackage{graphicx}
\usepackage{array, tabularx}
\usepackage{algorithm,algcompatible}
\usepackage{booktabs} 
\usepackage{paralist} 
\usepackage{verbatim} 
\usepackage{subfig} 
\usepackage{xcolor,marginnote,enumitem}
\usepackage[numbers]{natbib}
\numberwithin{equation}{section}

\theoremstyle{theorem}
\newtheorem{lemma}{Lemma}
\newtheorem{theorem}{Theorem}
\newtheorem{proposition}{Proposition}

\newtheorem{assumption}{Assumption}

\theoremstyle{remark}
\newtheorem{remark}{Remark}

\theoremstyle{definition}
\newtheorem{definition}{Definition}

\DeclareMathOperator{\dom}{dom}

\DeclareMathOperator*{\argmin}{arg\,min}

\DeclareMathOperator{\Div}{div}

\DeclareMathOperator{\TGV}{TGV}
\DeclareMathOperator{\BD}{BD}
\DeclareMathOperator{\BV}{BV}
\newcommand{\symgrad}{\mathcal{E}}
\newcommand{\mE}{\mathcal{E}}

\newcommand{\grad}{\nabla}

\newcommand{\lookUp}[1]{}

\newcommand{\norm}[2][]{\|{#2}\|_{#1}}

\newcommand{\scp}[3][]{\langle{#2},{#3}\rangle_{#1}}

\newcommand{\RR}{\mathbf{R}}

\newcommand{\mF}{\mathcal{F}}

\newcommand{\mI}{\mathcal{I}}

\newcommand{\mP}{\mathcal{P}}
\newcommand{\mV}{\mathcal{V}}
\newcommand{\mN}{\mathcal{N}}

\algnewcommand\INPUT{\item[\textbf{Input:}]}%
\algnewcommand\OUTPUT{\item[\textbf{Output:}]}%
\newcommand{\gap}{\mathfrak{G}}

\DeclareMathAlphabet{\mathbfit}{OML}{cmm}{b}{it}

\begin{document}

\title{An Efficient Augmented Lagrangian Method with Semismooth Newton Solver for Total Generalized Variation}

\author{Hongpeng Sun\thanks{Institute for Mathematical Sciences,
		Renmin University of China, 100872 Beijing, China.
		Email: \href{mailto:hpsun@amss.ac.cn}{hpsun@amss.ac.cn}.} }

\maketitle
\begin{abstract}
Total generalization variation (TGV) is a very powerful and important regularization for various inverse problems and computer vision tasks.
In this paper, we propose a semismooth Newton based augmented Lagrangian method for solving this problem.   The augmented Lagrangian method (also called as method of multipliers) is widely used for lots of smooth or nonsmooth variational problems. However, its efficiency heavily depends on solving  the corresponding  coupled and nonlinear system together and simultaneously. 
With efficient primal-dual semismooth Newton methods for the challenging and highly coupled nonlinear subproblems involving total generalized variation,  we develop a highly efficient and competitive augmented Lagrangian method  compared with some fast first-order method. With the analysis of the metric subregularities of the corresponding functions, we give both the global convergence and local linear convergence rate for the proposed augmented Lagrangian methods.
\end{abstract}

\paragraph{Key words.}{Augmented Lagrangian method, \and primal-dual semismooth Newton method, \and local linear convergence rate, \and metric subregularity}

\paragraph{AMS subject classifications.}
65K10, 	49J52,
49M15

\section{Introduction}

Total generalized variation (TGV) is an important regularization and image prior to various applications including medical imaging, computer vision, tomography, inverse problems in mathematical physics, and so on \cite{BPK, BNW, KNPR}. By including both the first 
and the second derivatives, TGV can overcome the staircase artifacts and bring out some advantages compared with total variation \cite{BPK}.  Here we mainly focus on the second-order TGV \cite{BPK}. 
Due to the complicated structure of the TGV \cite{BPK, BT}, the computation of the TGV regularized problem is usually very time-consuming and challenging. Currently, the first-order primal-dual method \cite{CP} is widely used. The fast iterative shrinkage-thresholding algorithm (FISTA) is employed in \cite{BPK} and the preconditioned Douglas-Rachford splitting method is also developed \cite{BS}.
There are first-order optimization methods. To the best of our knowledge, the second-order semismooth Newton method is first discussed in \cite{HPRS} with additional Tikhonov regularization on the dual variables.

In this paper, we are interested in the augmented Lagrangian method (abbreviated as ALM throughout this paper) originated by Hestenes \cite{HE} and Powell \cite{POW}.  ALM is very flexible for constrained optimization problems including both equality and inequality 
constraints \cite{BE,KK}.
It is a kind of bridge between first-order methods and second-order Newtons method. We refer to \cite{BE, FG, Roc2} for its early developments and \cite{BE, FG, KK} for the comprehensive and extensive studies on convex, nonsmooth, and variational optimization problems.   Furthermore, the convergence of ALM can be concluded in the general and powerful proximal point algorithm framework for convex optimization \cite{Roc1, Roc2}, due to the equivalence between ALM and the proximal point algorithm applying to the essential dual problem \cite{Roc1}.

However, it is challenging to solve the nonlinear and coupling systems simultaneously  while applying ALM. This is different from the alternating direction method of multipliers (ADMM) type methods \cite{GL1,FG}, which can decouple the unknown variables and update them consecutively like the Gauss-Seidel method.  For ALM, the extra effort is deserved if the nonlinear system can be solved efficiently. This is due to the appealing linear or asymptotic superlinear convergence of ALM with increasing step sizes \cite{Roc1, Roc2, LU}. We employ the semismooth Newton methods for the nonlinear subproblems of ALM, which already have lots of successful applications in semidefinite programming \cite{ZST}, compressed sensing \cite{LST, ZZST}, friction and contact problem \cite{Gor, Gor1} and total variation regularized imaging problems \cite{HK1,KK1}.

Currently, no attempt has been made to develop an ALM algorithm for TGV. In this paper, we propose a novel semismooth Newton based ALM for the TGV regularized image restoration problem. 
The proposed algorithm is based on applying ALM to the perturbed primal problem of TGV, where we can benefit from the strong convexity. With the ALM framework, we do not need the Tikhonov regularization on the dual variables as did in \cite{HPRS} for TGV regularized image restoration, where semismooth Newton method is applied directly to the corresponding optimality conditions.  The ALM can be seen as a kind of globalization of semismooth Newton methods, which directly aims at the perturbed original problem without Tikhonov regularizations on the dual variables. With ALM, the step sizes do not need to tend to infinity where the linear subproblems are quite ill-posed.  Our contributions belong to the following parts. First, by introducing some auxiliary variables, we use primal-dual semismooth Newton method \cite{HS} for the nonlinear system of ALM, which is very efficient without any globalization strategy including the Armijo line search experimentally. The proposed ALM is very efficient compared with some first-order algorithm such as the primal-dual method \cite{CP}, especially for very high accuracy tasks. 
Second, with the help of the calm intersection theorem \cite{KKU}, we also prove the novel metric subregularity of the maximal monotone operator associated with the dual problem under mild condition, which is more complicated compared with the TV (total variation) case \cite{SUNA} since both the primal and dual variables are highly coupled. The corresponding metric subregularity leads to the linear or asymptotic superlinear convergence rate of the dual sequence \cite{Roc1, Roc2, LU}. We can also obtain the asymptotic linear or superlinear convergence rate of the primal sequence for a certain case.

The rest of this paper is organized as follows. In section \ref{sec:tgvALM}, we give a brief introduction to TGV regularization and the ALM algorithm. In section \ref{sec:SSN}, we investigate the primal-dual semismooth Newton methods for ALM by introducing auxiliary variables, which turn out to be very efficient.  In section \ref{sec:conver:alm}, we analyze the metric subregularity for the maximal monotone operator associated with the corresponding dual problem. Together with the convergence of the semismooth Newton method, we get the corresponding asymptotic linear or superlinear convergence rate. In section \ref{sec:numer}, we present detailed numerical tests for all the algorithms including the comparison with some efficient  first-order algorithm. In  section \ref{sec:conclude}, we give some final conclusions. 

\section{TGV and augmented Lagrangian method} \label{sec:tgvALM}
In this section, we give a brief introduction of the TGV regularization and ALM together with some basic notations and terminology.
The $L^2$-TGV regularized image restoration model reads as follows \cite{BPK},
\begin{equation}
\label{eq:tgv_primal:o}
\min_{u \in \BV(\Omega)} \ F(u) + \TGV_{\alpha}^{2}(u), \quad \alpha =(\alpha_0, \alpha_1),
\end{equation}
with $\Omega$ denoting the image domain,  $F(u) = \norm[2]{Ku - f_0}^2/2+ \frac{\nu}{2}\|\nabla u\|_{2}^2$ representing the data fidelity term, $K$ being a linear and bounded operator, and $f_0$ being the noisy or degraded image.  $\TGV_{\alpha}^{2}(u)$ denotes the second order total generalized variation (TGV) regularization \cite{BPK} with positive  regularization parameters $\alpha_0$ and $\alpha_1$ \cite{BPK}.
Henceforth, we assume $-\nu \Delta + K^*K $ is positive definite  with $\nu= 0$ if $K^*K$ is positive definite and $\nu>0$ otherwise.
It is convenient that the TGV regularization can be reformulated as follows \cite{BS},
\begin{equation}\label{tgv:bv}
\TGV_{\alpha}^{2}(u) = \min_{w \in \BD(\Omega)} \ \alpha_{1} \| D u -w\|_{\mathcal{M}} + \alpha_{0} \|\mathcal{E}w\|_{\mathcal{M}},
\end{equation}
where $\BD(\Omega)$ denotes the space of vector fields of
{\it Bounded Deformation}, $Du$ represents the distributional
derivative being a vector-valued Radon measure (see Chapter 9.1 of \cite{KK}), and the weak symmetrized derivative $\mathcal{E}w
:= (\nabla w + \nabla w^T)/2$ is a matrix-valued Radon measure \cite{BPK}.  Actually, $Du$ essentially coincides with $\nabla u$ when $u$ is smooth. 
Moreover, $\|\cdot\|_{\mathcal{M}}$ denotes the Radon norm for the
corresponding vector-valued and matrix-valued Radon measures. The norm of $\BD(\Omega)$ is defined by
\begin{equation}
\|w\|_{\BD} :=\|w\|_{1} + \|\mathcal{E}w\|_{\mathcal{M}}.
\end{equation}

Throughout this paper, we will focus on the following perturbed and regularized primal problem in finite-dimensional spaces
\begin{equation}
\label{eq:tgv_primal}
\min_{u \in U, w\in V} \mathfrak{F}(u,w): =\ F(u) + \frac{a}{2}\|w\|_{2}^2 +  \alpha_{1} \| \nabla u -w\|_{1} + \alpha_{0} \|\mathcal{E}w\|_{1}, \tag{P}
\end{equation}
where $\frac{a}{2}\|w\|_{2}^2$ is a Tikhonov regularization term on $w$ with positive constant $a$ and $U$ or $V$ is the corresponding discrete space as follows \cite{BS}, i.e.,
\begin{equation}
U =  \{u: \Omega \rightarrow \RR \}, \quad V = \{w: \Omega \rightarrow \RR^2\}.
\end{equation}
Henceforth, all the space settings are finite-dimensional and all the operators and integrals are corresponding to the discrete settings. The motivation of the adding term  $\frac{a}{2}\|w\|_{2}^2$ is as follows. First, it will bring out strong convexity on the primal variable $(u,w)$ together with strong convexity of $F(u)$ on $u$, where various convex algorithms can be benefited. For example, the first-order primal-dual method can be accelerated \cite{CP} and the regularity of Newton derivatives for semismooth Newton methods can be guaranteed. Second, the adding term $\frac{a}{2}\|w\|_{2}^2$ can keep the quality of image  reconstructions as the original TGV regularization experimentally, which as will be shown in numerics.

The primal form \eqref{eq:tgv_primal} can be written as the following primal-dual form (see \cite{BS})
\begin{equation}
\label{eq:tgv-denoising-saddle:original}
\min_{\substack{u \in U, w \in V}} \max_{\substack{\lambda \in V, \mu \in W}} \mathfrak{ L}(u,w,\lambda,\mu), \
\end{equation}
where given indicator functions $\mI_{\{\norm[\infty]{\lambda} \leq \alpha_1\}}(\lambda)$  and $\mI_{\{\norm[\infty]{\mu} \leq \alpha_0\}}(\mu)$ \cite{BS},  $\mathfrak{L}$ is defined by
\begin{equation}\label{eq:frak:L}
\mathfrak{ L}(u,w,\lambda,\mu): =
\scp[2]{\grad u - w}{\lambda} + \scp[2]{\symgrad w}{\mu} + F(u) + \frac{a}{2}\|w\|_{2}^2-
\mI_{\{\norm[\infty]{\lambda} \leq \alpha_1\}}(\lambda) - \mI_{\{\norm[\infty]{\mu} \leq \alpha_0\}}(\mu).
\end{equation}
Here $\lambda \in H_{0}(\Div;\Omega)$ with vanishing Dirichlet boundary condition with the discrete space $V$.  $\mu \in \mathcal{C}_{c}(\Omega, \text{Sym}^2(\mathbb{R}^2))$ as in \cite{BPK} and the corresponding discrete space $W$ is defined as
\[
W = \{\mu :\Omega \rightarrow S^{2 \times 2}\}.
\]
Besides, the $L^2$ inner product and $\|\cdot\|_{2}^2$ in \eqref{eq:frak:L} are defined by
\[
\langle a, b \rangle_{2}: = \int_{\Omega} \langle a,b \rangle dx, \quad \|w\|_{2}^2: = \int_{\Omega} |w|^2 dx.
\]
The inner products $\langle \cdot, \cdot \rangle $ and the Euclid norm $|\cdot|$ above are defined as follows
\begin{align}
&\langle  u, v \rangle :=  u v, \quad u, v \in U, \quad |u|:=|u|_{2} = \sqrt{\langle  u, u \rangle },  \label{eq:inner:norms} \\
&\langle  \lambda, s \rangle := \lambda^{T}s =   \lambda^1s^1 +  \lambda^2s^2, \ \lambda = (\lambda^1,\lambda^2)^T, \ s = (s^1,s^2)^{T} \in V, \ |\lambda|:= |\lambda|_2 = \sqrt{\langle  \lambda, \lambda \rangle }, \notag \\
&\langle  \mu, r \rangle :=  \mu^1r^1 +  \mu^2r^2 + 2\mu^3r^3, \ \mu = (\mu^1,\mu^2, \mu^3)^T, \  r = (r^1,r^2,r^3)^{T} \in W, \ |\mu| := |\mu|_2= \sqrt{\langle  \mu, \mu \rangle }. \notag 
\end{align}
For $\mu\in W$, $\lambda \in V$, the discrete $\|\cdot \|_t$ norms with $1\leq t < \infty$ and $\|\cdot\|_{\infty}$  are
defined as follows,
\begin{align}
&\|\lambda\|_t = \Bigl( \sum_{(i,j) \in \Omega}
\bigl( (\lambda^1_{i,j})^2 + (\lambda^2_{i,j})^2 
\bigr)^{t/2}\Bigr)^{1/t}, \ \
\|\lambda\|_\infty = \max_{(i,j) \in \Omega} \ \sqrt{(\lambda^1_{i,j})^2 + (\lambda^2_{i,j})^2}, \label{eq:norms:pq} \\
&\|\mu\|_t = \Bigl( \sum_{(i,j) \in \Omega}
\bigl( (\mu^1_{i,j})^2 + (\mu^2_{i,j})^2 + 2(\mu^3_{i,j})^2
\bigr)^{t/2}\Bigr)^{1/t}, \ \
\|\mu\|_\infty = \max_{(i,j) \in \Omega} \ \sqrt{(\mu^1_{i,j})^2 + (\mu^2_{i,j})^2 +
	2(\mu^3_{i,j})^2}. \notag
\end{align}
%
Utilizing forward differences and its adjoint divergence are defined as follows \cite{BS}
\begin{equation}\label{eq:operators:adjoints}
\nabla  = \begin{bmatrix}
\partial_x^+ \\
\partial_y^+
\end{bmatrix}, \quad \langle \nabla u, \lambda \rangle_{2}  = \langle  u,  \nabla ^*\lambda \rangle_{2}, \quad  \nabla^* = -\Div, \quad \Div  =  [\partial_x^-, \partial_y^-].
\end{equation}
The symmetrized derivative thus can be defined as follows
\[
\symgrad w =
\begin{bmatrix}
\partial_x^+ w^1 & \tfrac12(\partial_y^+ w^1
+ \partial_x^+ w^2) \\
\tfrac12(\partial_y^+ w^1
+ \partial_x^+ w^2) &
\partial_y^+ w^2
\end{bmatrix}
:=
\begin{bmatrix}
\partial_x^+ w^1 \\
\partial_y^+ w^2 \\
\tfrac12(\partial_y^+ w^1
+ \partial_x^+ w^2)
\end{bmatrix}
\]
where the second equation is understood in terms of the
identification $W = U^3$ \cite{BS,BT}. Consequently, the negative adjoint realizes
a discrete negative divergence operator according to
$\scp[W]{\symgrad w}{\mu} = \scp[V]{ w}{\symgrad^*\mu}=-\scp[V]{w}{\Div \mu}$ for all $w \in V$, $\mu \in W$,
leading to
\[
\Div \mu =
\begin{bmatrix}
\partial_x^- \mu^1 + \partial_y^- \mu^3 \\
\partial_x^- \mu^3 + \partial_y^- \mu_2
\end{bmatrix}, \quad \mE^* = -\Div.
\]
Actually, by the Fenchel-Rockafellar duality theory \cite{HBPL,KK}, the dual problem of \eqref{eq:tgv_primal} becomes
\begin{subequations}\label{eq:dual:tgv} 
	\begin{align}
	\max_{\lambda\in V, \mu\in W }  -\mathfrak{D}(\lambda,\mu):= &-\frac{1}{2}\|\Div \lambda + K^*f_0\|_{H^{-1}}^2 + \frac{1}{2} \|f_0\|_2^2  - \frac{1}{2a} \|\lambda-\mE^* \mu \|_2^2 \\
	&-\mI_{\{\norm[\infty]{\lambda} \leq \alpha_1\}}(\lambda) - \mI_{\{\norm[\infty]{\mu} \leq \alpha_0\}}(\mu),  
	\end{align}
\end{subequations}
where $H := K^*K - \mu \Delta$ being positive definite.
By the first-order optimality (KKT) conditions of \eqref{eq:frak:L} (e. g., see the  \cite{HS} (Theorem 2.1)), the solution $(\bar u, \bar w)$ of the primal problem \eqref{eq:tgv_primal} and the dual solutions $(\bar \lambda, \bar \mu)$ of \eqref{eq:dual:tgv} have the following relations
\begin{subequations}\label{eq:optimalites:tgv:pd}
	\begin{align}
	H\bar u - \Div \bar \lambda &= K^*f_0, \\
	a\bar w-\bar \lambda + \mE^* \bar  \mu&=0,\\
	-\alpha_1( \nabla \bar u- \bar w) + |\nabla \bar u -\bar w|\bar \lambda &= 0, \quad \text{if} \quad |\bar \lambda| = \alpha_1, \\ 
	\nabla \bar u -\bar w& = 0, \quad \text{if} \quad  |\bar \lambda| < \alpha_1, \\
	-\alpha_0 \mE  \bar w + |\mE \bar w|\bar \mu &= 0, \quad \text{if} \quad  |\bar \mu| = \alpha_0,  \\ 
	\mE \bar w& = 0, \quad \text{if} \quad  |\bar \mu| < \alpha_0.
	\end{align}
\end{subequations}

For employing ALM, let  us introduce the following auxiliary variables $h_1 \in V$, $h_2 \in W$ 
\[
h_1 = \nabla u - w, \quad h_2 = \mE w,
\]
and the corresponding Lagrangian multipliers $\lambda \in V$, $\mu \in W$. Introducing the step size $\sigma$, we define the augmented Lagrangian function for \eqref{eq:tgv_primal},
\begin{align}
L_{\sigma}(u,w,h_1, h_2;\lambda, \mu): =& F(u)  +\frac{a}{2}\|w\|_{2}^2 + \alpha_1 \|h_1\|_{1} + \alpha_0 \|h_2\|_{1} + \langle \lambda, \nabla u -w -h_1 \rangle_2 \notag \\
&+ \langle \mu, \mE w - h_2 \rangle_2 + \frac{\sigma}{2}\|\nabla u-w -h_1\|_{2}^2 + \frac{\sigma}{2}\| \mE w - h_2\|_{2}^2.  \label{eq:aug:lag:uw1}
\end{align}
With the augmented Lagrangian $ L_{\sigma}(u,w,h_1, h_2;\lambda, \mu)$, given $\lambda^0$, $\mu^0$ and $\sigma_0$, the classical augmented Lagrangian method for solving \eqref{eq:tgv_primal} can be written as follows \cite{BE,FG,GL1}, for $k=0, 1\cdots, $
\begin{align}
(u^{k+1}, w^{k+1}, h_1^{k+1}, h_{2}^{k+1}) &:= \argmin_{u,w,h_1, h_2}  L_{\sigma_k}(u,w,h_1, h_2;\lambda^k, \mu^k), \label{eq:update:primals:alm}\\
\lambda^{k+1} &:= \lambda^k + \sigma_k(\nabla u^{k+1} - w^{k+1}-h_1^{k+1}),\label{eq:update:lambda} \\
\mu^{k+1} &:= \mu^k + \sigma_k(\mE w^{k+1} -h_2^{k+1}),\label{eq:update:mu} \\
\sigma_{k+1} &\geq \sigma_k >0, \quad \sigma_{k} \rightarrow \sigma_{\infty} < +\infty.
\end{align}
For fixed $\lambda^k$, $\mu^k$ and $\sigma_k$, with direct calculations, the optimality  conditions of \eqref{eq:update:primals:alm} can be written as
\begin{subequations}\label{eq:opti:h}
	\begin{align}
	&H u- f + \nabla^*[ {\lambda^k} + \sigma_k(\nabla u-w - h_1)]=0, \\
	&aw - [{\lambda^k} + \sigma_k(\nabla u-w - h_1)] + \mE^*[\sigma_k (\mE w-h_2) + \mu^k]=0, \\
	& h_1 = (I + \frac{\alpha_1}{\sigma_k} \partial \|\cdot\|_1)^{-1}(\frac{\lambda^k}{\sigma_k}+ \nabla u -w),\\
	& h_2 = (I + \frac{\alpha_0}{\sigma_k} \partial \|\cdot\|_1)^{-1}(\frac{\mu^k}{\sigma_k}+ \mE w),
	\end{align}
\end{subequations}
where $f: = K^* f_0$ henceforth. Semismooth Newton methods \cite{KK} can be employed directly to solve such kind of nonlinear equation \eqref{eq:opti:h}. However, we will introduce another equivalent nonlinear system compared to \eqref{eq:opti:h} through new auxiliary variables, which is more convenient for semismooth Newton solvers as shown for some variants of TV model \cite{SUNA,HS} and TGV model \cite{HPRS}.
Actually, by the Moreau's equality \cite{HBPL,CP} 
\begin{equation}\label{eq:moreau:indentity}
x = (I + \tau \partial G)^{-1}(x) + \tau (I  +\frac{1}{\tau} \partial G^*)^{-1}(\frac{x}{\tau}),
\end{equation}
and the notations $G_1(\lambda):= \mI_{\{\norm[\infty]{\lambda} \leq \alpha_1\}}(\lambda)$, $G_2(\mu):= \mI_{\{\norm[\infty]{\mu} \leq \alpha_0\}}(\mu)$, we arrive at
\begin{subequations}\label{eq:moreau:sub12}
	\begin{align}
	&{\lambda^k}+ {\sigma_k}(\nabla u -w) - \sigma h_1 = (I + \sigma_k \partial G_1)^{-1}({\lambda^k}+ {\sigma_k}(\nabla u -w)), \\
	&  {\mu^k}+ {\sigma_k} (\mE w -h_2) = (I + \sigma_k \partial G_2)^{-1}({\mu^k}+ {\sigma_k}\mE w).
	\end{align}
\end{subequations}
Now, let  us introduce the projections for arbitrary $w \in V$ and $l \in W$ with $\mP_{\alpha_1}(w):= (I + \sigma_k \partial G_1)^{-1}(w)$ and $\mP_{\alpha_0}(l):= (I + \sigma_k \partial G_2)^{-1}(l)$, i.e., 
\begin{equation}\label{eq:projections:def}
\mP_{\alpha_1}(w):= {w}/{\max(1.0, {|w|}/{\alpha_1})},  \ \mP_{\alpha_0}(l):=  {l}/{\max(1.0, {|l|}/{\alpha_0})}, 
\end{equation}
which are understood in pointwise sense. Let us also denote
\begin{subequations}\label{eq:moreau:sub12:pq}
	\begin{align}
	p: &= \mP_{\alpha_1}({\lambda^k}+ {\sigma_k}(\nabla u -w))= \dfrac{{\lambda^k}+ {\sigma_k}(\nabla u -w)}{\max(1.0, \frac{|{\lambda^k}+ {\sigma_k}(\nabla u -w)|}{\alpha_1})},  \\
	q: &= \mP_{\alpha_0}({\mu^k}+ {\sigma_k}\mE w)=\dfrac{{\mu^k}+ {\sigma_k}\mE w}{\max(1.0, \frac{|{\mu^k}+ {\sigma_k}\mE w|}{\alpha_0})}.
	\end{align}
\end{subequations}
With the auxiliary variables $p$ and $q$ in \eqref{eq:moreau:sub12:pq} with notation $x := (u,w,p,q)^T$, the optimality conditions in \eqref{eq:opti:h} then becomes
\begin{equation}\label{eq:opti:pq}
\mathcal{F}(x) = 0, \ \ 
\mathcal{F}(x) := \begin{bmatrix}
Hu-f + \nabla^*p \\
aw - p + \mE^*q  \\
-({\lambda^k}+ {\sigma_k}(\nabla u -w))+\max(1.0, \frac{|{\lambda^k}+ {\sigma_k}(\nabla u -w)|}{\alpha_1}) p \\
- ({\mu^k}+ {\sigma_k}\mE w) + \max(1.0, \frac{|{\mu^k}+ {\sigma_k}\mE w|}{\alpha_0})q 
\end{bmatrix}.
\end{equation}
Henceforth, we will focus on semismooth Newton methods for solving \eqref{eq:opti:pq} instead of \eqref{eq:opti:h}. Our motivation is that the primal-dual semismooth Newton method to be discussed can be highly efficient for the nonlinear system \eqref{eq:opti:pq} compared to semismooth Newton method for solving \eqref{eq:opti:h}, which is also shown in \cite{SUNA} for TV regularization problems. Besides, experimentally, no line search techniques including Armijo line search are needed as will be shown  in numerics. 

For the updates of the Lagrangian multipliers $\lambda^{k+1}$ and $\mu^{k+1}$, with  \eqref{eq:update:lambda}, \eqref{eq:update:mu}, and \eqref{eq:moreau:sub12}, the updates during each ALM iteration after \eqref{eq:update:primals:alm} can alternatively be 
\begin{subequations}\label{eq:update:multi:pqway}
	\begin{align}
	\lambda^{k+1}&=\lambda^k+\sigma_k(\nabla u -\sigma_k w) -\sigma_k h_1 = (I + \sigma_k \partial G_1)^{-1}({\lambda^k}+ {\sigma_k}(\nabla u -w))=p, \\
	\mu^{k+1}&=\mu^k+\sigma_k \mE w -\sigma_k h_2 = (I + \sigma_k \partial G_2)^{-1}({\mu^k}+ {\sigma_k}\mE w)=q,
	\end{align}
\end{subequations}
which are nonlinear updates compared to the linear updates \eqref{eq:update:lambda} and \eqref{eq:update:mu}.  We refer to \cite{KK} (chapter 4) for general nonlinear updates of Lagrangian multipliers with more general derivations and variants of ALM. 

\section{Semismooth Newton method and Newton derivative}\label{sec:SSN}
In this section, we  will look closely at the primal-dual semismooth Newton method and its delicate application to \eqref{eq:opti:pq} along with the ALM for the primal problem \eqref{eq:tgv_primal}. For semismooth Newton methods, the Newton derivative is of critical importance.
The following definition of Newton derivative is originally for Banach spaces \cite{KK}, which is also applicable for our finite-dimensional space cases.
\begin{definition}\label{def:newton:deri}[Newton differentiable and Newton Derivative \cite{KK}] $F: D \subset X \rightarrow Z$ is called Newton differentiable at $x$ if there exist an open neighborhood $N(x) \subset D$ and mapping $\mV: N(x)\rightarrow \mathcal{L}(X,Z)$ such that (Here the spaces $X$ and $Z$ are Banach spaces.)
	\begin{equation}
	\lim_{|h|\rightarrow 0} \frac{|F(x+h)-F(x)-\mV(x+h)h|_{Z}}{|h|_X}=0.
	\end{equation}
	The family $\{\mV(s): s \in N(x)\}$ is called a Newton derivative of $F$ at $x$.
\end{definition}
If $F:\mathbb{R}^n \rightarrow \mathbb{R}^m$ and the set of mapping $\mV$ is Clarke's generalized gradient (or Clarke's generalized derivative) $\partial_{C} F$ \cite{CL}, we call $F$ is semismooth \cite{KKU} as in the following definition.
\begin{definition}[Semismoothness \cite{RM, LST, MU}] Let $F: O \subseteq X \rightarrow Y$ be a locally Lipschitz continuous function on the open set $O$, where $X$ and $Y$ are finite-dimensional Hilbert spaces.  $F$ is said to be semismooth at $x \in O$ if $F$ is directionally differentiable at $x$ and for any $\mV\in \partial_{C} F(x+ \Delta x)$ with $\Delta x \rightarrow 0$,
	\[
	F(x+\Delta x) -F(x) -\mV\Delta x = {o}(\|\Delta x\|).
	\]  
\end{definition}
We thus can choose the element of Clark's generalized derivative as the  Newton derivative for semismooth functions. The Newton derivatives of vector-valued functions can be computed component wisely \cite{Cla} (Theorem 9.4). Together with the definition of semismoothness, we have the following lemma.  
\begin{lemma}\label{lem:vector:semismooth:newton}
	Suppose $F : \mathbb{R}^n \rightarrow \mathbb{R}^m$ and $F = (F_1(x), F_2(x), \cdots, F_l(x))^{T}$ with $F_i : \mathbb{R}^n \rightarrow \mathbb{R}^{l_i} $  being semismooth. Here $l_i \in \mathbb{Z}^{+}$ and $\sum_{i=1}^l l_i=m$. Denoting the Newton derivative of $F_i(x)$ as $ D_N F_i(x)$ by definition \ref{def:newton:deri} and assuming $ D_N F_i(x) \in \partial_{C} F_i(x)$, $i  = 1,2,\cdots, l$, then the Newton derivative of $F$ can be chosen as 
	\begin{equation}
	D_N F(x) = \begin{bmatrix}
	D_N F_1(x), D_N F_2(x)
	, \cdots, D_N F_l(x)	\end{bmatrix}^T.
	\end{equation}
\end{lemma}
Once the Newton derivative is obtained, the semismooth Newton method for the nonlinear equation $F(x)=0$ can be written as 
\begin{equation}\label{semi:smoothnewton:sys:f}
x^{l+1} = x^{l} -\mV(x^l)^{-1} F(x^l)  \Rightarrow \mV(x^l)x^{l+1} = \mV(x^l)x^l  - F(x^l),
\end{equation}
where $\mV(x^l) \in \partial_{C} F(x^l)$ is the semismooth Newton derivative of $F$ at $x^l$, and $\mV(x)^{-1}$ exist and are uniformly bounded for all $x$ in a small neighborhood of the solution $x^*$ of $F(x^*)=0$.
For the convergence rate of semismooth Newton methods, we have the following proposition.
\begin{proposition}[Superlinear Convergence \cite{KK}] 
	Suppose $x^*$ is a solution to $F(x)=0$ and $F$ is Newton differentiable at $x^*$ with Newton derivative $\mV$. If $\mV$ is nonsingular for all $x \in N(x^*)$ and $\{ \|\mV(x)^{-1}\| : x \in N(x^*)\}$ is bounded ($ N(x^*)$ is a  neighborhood of $x^*$), then the Newton iteration
	\[
	x^{l+1} = x^l-\mV(x^l)^{-1}F(x^l), 
	\]
	converges superlinearly to $x^*$ provided that $|x^0-x^*|$ is sufficiently small.
\end{proposition}

%

Now we turn to the semismoothness of nonlinear system \eqref{eq:opti:pq}. The only nonlinear or nonsmooth parts come from the function $\Upsilon(w):=\max (1.0, |\mu^k + \sigma_k \mE w|/\alpha_0)$ and $\Pi(u,w):=\max(1.0,{|{\lambda^k}+ {\sigma_k}(\nabla u^l -w^l)|}/{\alpha_1})$.
\begin{lemma}\label{lem:semismooth:max}
	The function $\Upsilon(w): = \max (1.0, \dfrac{|\mu^k + \sigma_k \mE w|}{\alpha_0})$  is semismooth at $w$ and its Clarke's generalized gradient for $w$, i.e.,  $\partial_{C}\Upsilon(w)$ is as follows,
	\begin{equation}\label{eq:newton:deri:max:iso}
	\left\{\chi^s_{w}\dfrac{\sigma_k}{\alpha_0}\dfrac{\langle \mu^k + \sigma_k \mE w,  \mE \cdot \ \rangle }{ |\mu^k + \sigma_k \mE u |} \ | \ s\in[0,1]\right\} = \partial_{C}\Upsilon(w),
	\end{equation}
	where $\chi^s_{w}$ is an element of the Clarke's generalized derivatives of $\max(\cdot, 1.0)$ defined by,
	\begin{equation}\label{eq:newton:deri:max:w}
	\chi^s_{w} := \begin{cases}
	1, \quad & |\mu^k +\sigma_k \mE w | /\alpha_0 >1.0, \\
	s, \quad & |\mu^k +\sigma_k \mE w | /\alpha_0 =1.0, \  s\in [0,1], \\ 
	0, \quad & |\mu^k  + \sigma_k \mE w | / \alpha_0  <1.0.
	\end{cases}
	\end{equation}
\end{lemma}
\begin{proof}
	We will mainly prove that $\Upsilon(w)$ is a $PC^{\infty}$ (piecewise smooth) function  of $w$ \cite{Sch}. 
	It is thus semismooth on $w$ (see \cite{MU}, Proposition 2.26).
	Introduce $\Upsilon_1(w) = 1.0$ and $\Upsilon_2(w) = {|\mu^k + \sigma_k \mE w|}/{\alpha_0}$ which are \emph{selection functions} of $\Upsilon(w)$ and $\Upsilon(w)$ is \emph{continuous selection} of the functions $\Upsilon_1(w)$ and $\Upsilon_2(w)$ \cite{Sch} (Chapter 4) (or Definition 4.5.1 of \cite{FP}). 
	Noting $\Upsilon_1(w)$ is smooth function and $\Upsilon_2(w)$ is smooth in any open set outside the closed set $S_0:=\{w \ | \ |\mu^k + \sigma_k \mE w|=0\}$,  there thus exists a small open neighborhood of $w$ such that $\Upsilon_1(w)$ and $\Upsilon_2(w)$ are smooth functions for any $w\in S_{\alpha_0}:=\{w \ | \ |\lambda^k + \mu_k \mE w|=\alpha_0\}$.   $\Upsilon(w)$ is thus a $PC^{\infty}$ function of $w$ outside $S_0$.  Furthermore, we have
	\[
	\nabla_w \Upsilon_1(w) = 0, \quad  \nabla_w \Upsilon_2(w) = {\sigma_k\langle \mu^k + \sigma_k \mE w,  \mE \cdot \ \rangle }/{ (\alpha_0|\mu^k + \sigma_k \mE  w|)}, \quad \forall w \notin S_0.
	\]
	For any $w \in S_{\alpha_0}$, by \cite{Sch} (Proposition 4.3.1), we thus see
	\[
	\partial_{C}\Upsilon(w) = \text{co}\{\nabla_w\Upsilon_1(w), \nabla_w \Upsilon_2(w)\},
	\]
	where ``$\text{co}$" denotes the convex hull of the corresponding sets.
\end{proof}
Similarly, denoting the Clarke's generalized derivative of $\Pi(u,w)$ on $u$ (or $w$) by $\partial_{C,u}\Pi(u, w)$ (or $\partial_{C,w}\Pi(u, w)$), we have the following lemma. 
\begin{lemma}\label{lem:semismooth:max:uw}
	The function $\Pi(u, w): = \max(1.0, \dfrac{|{\lambda^k}+ {\sigma_k}(\nabla u -w)|}{\alpha_1})$  is semismooth at $u$ or $w$ and the Clarke's generalized gradient for $u$ or $w$ is as follows,
	\begin{align}
	&\left\{\chi^s_{u,w}\dfrac{\sigma_k}{\alpha_1}\dfrac{\langle {\lambda^k}+ {\sigma_k}(\nabla u -w),  \nabla \cdot \ \rangle }{ |{\lambda^k}+ {\sigma_k}(\nabla u -w)|} \ | \ s\in[0,1]\right\} = \partial_{C,u}\Pi(u, w), \label{eq:newton:deri:max:iso:uw:u} \\
	&\left\{\chi^s_{u,w}\dfrac{\sigma_k}{\alpha_1}\dfrac{\langle {\lambda^k}+ {\sigma_k}(\nabla u -w),  -I \cdot \ \rangle }{ |{\lambda^k}+ {\sigma_k}(\nabla u -w)|} \ | \ s\in[0,1]\right\} = \partial_{C,w}\Pi(u, w), \label{eq:newton:deri:max:iso:uw:w}
	\end{align}
	where $\chi^s_{u,w}$ is an element of the Clarke's generalized derivatives of $\max(\cdot, 1.0)$ defined by,
	\begin{equation}\label{eq:newton:deri:max:ww}
	\chi^s_{u,w} = \begin{cases}
	1, \quad & |{\lambda^k}+ {\sigma_k}(\nabla u -w)| /\alpha_1 >1.0, \\
	s, \quad & |{\lambda^k}+ {\sigma_k}(\nabla u -w)| /\alpha_1 =1.0, \  s\in [0,1] \\ 
	0, \quad & |{\lambda^k}+ {\sigma_k}(\nabla u -w)| / \alpha_1  <1.0.
	\end{cases}
	\end{equation}
\end{lemma}

For the nonlinear equation \eqref{eq:opti:pq}, since each  component of $\mF$ is affine function of $p$ or $q$, $\mF$ is semismooth on $p$ or $q$. Together with Lemma \ref{lem:semismooth:max} and \ref{lem:semismooth:max:uw}, we thus conclude that each component of $\mF$ is semismooth on $x$. The semismoothness of $\mF$ on $x$ then follows \cite{MU} (Proposition 2.10). Now let  us turn to the semismooth Newton derivative of $\mF$.  Henceforth, we choose the Newton derivatives of $\Upsilon(w)$ on $w$ and $\Pi(u,w)$ on $u$ or $w$ by choosing $s=1$ in \eqref{eq:newton:deri:max:w}, \eqref{eq:newton:deri:max:iso:uw:u} and \eqref{eq:newton:deri:max:iso:uw:w} with the notations $\chi_{u,w}:=\chi^1_{u,w}$ and 	$\chi_{w}:=\chi^1_{w}$. Since $\nabla_{p}(\Pi(u,w)p) = \Pi(u,w)$, $\nabla_{q}(\Upsilon(w)q) = \Upsilon(w)$, and with Lemma \ref{lem:semismooth:max} and \ref{lem:semismooth:max:uw}, we arrive at
\begin{align}
&\chi_{u,w} \dfrac{\sigma_k}{\alpha_1}\dfrac{\langle {\lambda^k}+ {\sigma_k}(\nabla u -w),  \nabla \cdot \ \rangle }{ |{\lambda^k}+ {\sigma_k}(\nabla u -w)|} p \in \partial_{C,u}(\Pi(u,w)p), \ \ 	
\chi_{w}\dfrac{\sigma_k}{\alpha_0}\dfrac{\langle \mu^k + \sigma_k \mE w,  \mE \cdot \ \rangle }{ |\mu^k + \sigma_k \mE u |} q \in \partial_{C,w}(\Upsilon(w)q), \notag \\
&\chi_{u,w}\dfrac{\sigma_k}{\alpha_1}\dfrac{\langle {\lambda^k}+ {\sigma_k}(\nabla u -w),  -I \cdot \ \rangle }{ |{\lambda^k}+ {\sigma_k}(\nabla u -w)|}p \in \partial_{C,w}(\Pi(u, w)p), \label{eq:clarke:deri:compli}
\end{align}
where $\partial_{C,w}$ denotes the Clark's generalized gradient at $w$. Denote $x^l := (u^l, w^l, q^l, p^l)^{T}$.
With \eqref{eq:clarke:deri:compli}, noting all terms in each component of $\mF$ defined by \eqref{eq:opti:pq} except $\Pi(u,w)p$ and $\Upsilon(w)q$ are affine functions on $x$, the Newton derivative $\mV(x^l) \in \partial_{x}\mF(x)|_{x=x^l}$ can thus be chosen as
\begin{equation}\label{eq:newton:deri:b}
\mV(x^l) = \begin{bmatrix}
A & B \\
C_l&D_l
\end{bmatrix}, \quad \text{where} \quad A  = \begin{bmatrix}
U_{\mu} &0\\
0&W_{\alpha}
\end{bmatrix}, \quad
B = \begin{bmatrix}
\nabla^* & 0 \\
-I & \mE^*
\end{bmatrix},
\end{equation}
with notations 
\[
U_{\mu}: = H =K^*K - \mu \Delta,  \quad W_{a}: = a I, 
\]
and $C_l$, $D_l$ are the following operator matrices 
\begin{equation}\label{eq:AB:original}
D_l= \begin{bmatrix}
\max(1.0, \dfrac{|{\lambda^k}+ {\sigma_k}(\nabla u^l -w^l)|}{\alpha_1})&0 \\
0&\max(1.0, \dfrac{|{\mu^k}+ {\sigma_k}\mE w^l|}{\alpha_0})
\end{bmatrix},
\end{equation}
\begin{equation}\label{eq:C:original}
C_l = \begin{bmatrix}
-\sigma_k \nabla + \chi_{u^l,w^l}\frac{\sigma_k}{\alpha_1}\frac{\langle {\lambda^k}+ {\sigma_k}(\nabla u^l -w^l), \nabla \cdot \rangle }{|{\lambda^k}+ {\sigma_k}(\nabla u^l -w^l)|}p^l& \sigma_k I + \chi_{u^l,w^l}\frac{\sigma_k}{\alpha_1}\frac{\langle \sigma_k( w^l-\nabla u^l)-\lambda^k,  \cdot \rangle }{|{\lambda^k}+ {\sigma_k}(\nabla u^l -w^l)|}p^l\\
0&-\sigma_k \mE + \chi_{w^l}\frac{\sigma_k}{\alpha_0}\frac{\langle {\mu^k}+ {\sigma_k}\mE w^l,  \mE \cdot \rangle}{|{\lambda^k}+ {\sigma_k}\mE w^l| }q^l
\end{bmatrix}.
\end{equation}
The Newton update becomes
\begin{equation}\label{semi:smoothnewton:sys}
\mV(x^l)x^{l+1} = \mV(x^l)x^l  - \mF(x^l)=[b_1^l, b_2^l]^T,
\end{equation}
where
\begin{align}
&b_1^l = \begin{bmatrix}
f \\0
\end{bmatrix},\quad
b_2^l = \begin{bmatrix}
\lambda^k \\ \mu^k
\end{bmatrix} + 
\begin{bmatrix}
\dfrac{\sigma_k}{\alpha_1}\chi_{u^l,w^l}\dfrac{\langle {\lambda^k}+ {\sigma_k}(\nabla u^l -w^l), \nabla u^l-w^l \rangle }{|{\lambda^k}+ {\sigma_k}(\nabla u^l -w^l)|}p^l\\
\chi_{w^l}\dfrac{\sigma_k}{\alpha_0}\dfrac{\langle {\mu^k}+ {\sigma_k}\mE w^l,  \mE w^l \rangle}{|{\mu^k}+ {\sigma_k}\mE w^l| }q^l
\end{bmatrix}.
\end{align}
However, it is not necessary to solve $x^{l+1}$ in \eqref{semi:smoothnewton:sys} directly . We will employ the Schur complement $\mV(x^l)/D_l$ or $\mV(x^l)/A$ instead. For the  Schur complement $\mV(x^l)/D_l$, by direct calculation, we obtain
 the equation of $(u^{l+1}, w^{l+1})$ 
\begin{equation}\label{eq:tgv:primaluw:eq}
(A - BD_l^{-1}C_l)[u^{l+1}, w^{l+1}]^T =b_1^l- BD_l^{-1}b_2^l. 
\end{equation}
After calculating the update $(u^{l+1}, w^{l+1})$, we can get $(p^{l+1}, q^{l+1})$  through
\begin{equation}\label{eq:update:pd:after:uw}
\begin{bmatrix}
p^{l+1}\\q^{l+1}
\end{bmatrix}
= D_l^{-1}\left(b_2^l - C_l 
\begin{bmatrix}
u^{l+1}\\w^{l+1}
\end{bmatrix}\right).
\end{equation}
Actually, we can also calculate $(p^{l+1}, q^{l+1})$ first through the Schur complement $\mV(x^l)/A$, i.e.,
\begin{equation}\label{eq:ssn:dual}
(D_l - C_lA^{-1}B)(p^{l+1}, q^{l+1})^T = b_2^l - C_lA^{-1}b_1^l.
\end{equation} 
The update of $(u^{l+1}, w^{l+1})$ thus follows 
\begin{equation}\label{eq:updateuw:after:pq}
(u^{l+1}, w^{l+1})^T = A^{-1}(b_1^l - B(p^{l+1}, q^{l+1})^T).
\end{equation} 

In fact, for the equation \eqref{eq:tgv:primaluw:eq}, we would show that the operator in \eqref{eq:tgv:primaluw:eq} is always positive definite with $(p^l, q^l)$ belonging to $\{p: \norm[\infty]{p} \leq \alpha_1\} \times \{q: \norm[\infty]{q} \leq \alpha_0\}$. Let us denote $x_1:=(u,w)^T \in U\times V$ and $x_2:=(p,q)^T \in V\times W$ with $x=(x_1, x_2)$ as before.

\begin{theorem}\label{thm:posi:primal}
	For the semismooth Newton update \eqref{eq:tgv:primaluw:eq}, the Schur complement $\mV(x^l)/D_l : = (A-BD_{l}^{-1}C_l)$ is positive definite for $l \in \mathbb{N}$ under conditions that $\norm[\infty]{p^l} \leq \alpha_1$ and $\norm[\infty]{q^l} \leq \alpha_0$. Furthermore, the sequence $\{\mV(x^l)/D_l, \ \ l\in \mathbb{N} \}$ is uniformly bounded, since we have
	\begin{equation}
	\langle (A-BD_{l}^{-1}C_l) x_1, x_1\rangle  \geq  \|x_1\|_{A}^2=\|u\|_{H}^2 + \|w\|_{aI}^2, \quad \forall x_1: =(u,w)^T \in U\times V.
	\end{equation}
\end{theorem}

\begin{proof}
	Since $K^*K - \mu \Delta > 0$ and $a >0$, we see $A$ is positive definite. There thus exists a constant $c_0$, such that
	\begin{equation}\label{eq:positive:A}
	A \geq c_0 I.
	\end{equation}
	We would show that the operator $-BD_{l}^{-1}C_l$ is positive semidefinite. 	
	For the positive semidefiniteness of $-BD_k^{-1} C_{k}$, we just need to prove that for any $x_1$,
	\[
	\langle x_1, -BD_l^{-1} C_{l} x_1 \rangle \geq 0.
	\]
	Indeed, with \eqref{eq:AB:original} and \eqref{eq:C:original}, we have
	\begin{align}
	&\langle x_1, -BD_l^{-1} C_{l} x_1 \rangle =  \bigg \langle \begin{pmatrix}  \nabla u -w\\ \mE w \end{pmatrix}, D_{l}^{-1}
	\begin{pmatrix}
	(\nabla u-w)  - \frac{\chi_{u^l,w^l}}{\alpha_1} \frac{\langle \nabla u^l -w^l, \nabla u -w \rangle }{|\nabla u^l -w^l|}p^l \\
	\mE w - \frac{\chi_{w^l}}{\alpha_0} \frac{\langle \mE w^l, \mE w \rangle}{|\mE w^l|} q^l
	\end{pmatrix}
	\bigg \rangle  \notag \\
	& = \frac{1}{\max(1.0, \frac{|{\lambda^k}+ {\sigma_k}(\nabla u^l -w^l)|}{\alpha_1})} \left[   |\nabla u -w|^2  - \langle \nabla u -w, p^l \rangle \frac{\chi_{u^l,w^l}}{\alpha_1} \frac{\langle \nabla u^l -w^l, \nabla u -w \rangle }{|\nabla u^l -w^l|}      \right] \notag \\
	&\ \ + \frac{1}{\max(1.0, \frac{|{\mu^k}+ {\sigma_k}\mE w^l|}{\alpha_0})} \left[  |\mE w|^2 - \langle \mE w, q^l  \rangle  \frac{\chi_{w^l}}{\alpha_0}\frac{\langle \mE w^l, \mE w \rangle }{|\mE w^l|}  \right].\label{eq:estimate:semi}
	\end{align}
	Remembering that during all the semismooth Newton iterations \eqref{eq:tgv:primaluw:eq}, we have the conditions
	\[
	|p^l| \leq \alpha_1, \quad |q^l| \leq \alpha_0.
	\]
	It is straightforward that
	\begin{subequations}\label{eq:estimate:semi:inner}
		\begin{align}
		&\langle \nabla u -w, p^l \rangle \frac{\chi_{u^l,w^l}}{\alpha_1} \frac{\langle \nabla u^l -w^l, \nabla u -w \rangle }{|\nabla u^l -w^l|}
		\leq |\nabla u -w|^2, \\
		&\langle \mE w, q^l \rangle  \frac{\chi_{w^l}}{\alpha_0} \frac{\langle \mE w^l, \mE w \rangle }{|\mE w^l|}  \leq |\mE w|^2.
		\end{align}
	\end{subequations}
	Combining \eqref{eq:estimate:semi} and \eqref{eq:estimate:semi:inner}, we conclude the positive semidefiniteness of $-BD_{l}^{-1}C_l$. Together with \eqref{eq:positive:A}, we get this theorem.	
\end{proof}
\begin{remark}\label{rem:pro:feasible}
	Theorem \ref{thm:posi:primal} tells that if we can keep the constraints  $\norm[\infty]{p^l} \leq \alpha_1$ and $\norm[\infty]{q^l} \leq \alpha_0$ during each Newton iteration, $A-BD_{l}^{-1}C_l$ in \eqref{eq:tgv:primaluw:eq} would have a uniform low bound and is thus well-conditioned for fixed $\sigma$. This certainly can benefit iterative solvers including BiCGSTAB \cite{VAN} for solving \eqref{eq:tgv:primaluw:eq}. The constraints can be satisfied by projections to the corresponding feasible sets. This kind of strategy is inspired by semismooth Newtons directly applied to the TGV model \cite{HPRS} or TV model \cite{HS}, where Tikhonov regularization on dual variables is employed.
\end{remark}

Now, let  us turn to ${\mV}(x^l)^{-1}$ and the Schur complement $\mV(x^l)/A$.  
With \eqref{eq:AB:original}, it can be seen that $D_{l} \geq I$ and $D_{l}^{-1} \leq I$ follows. For the regularity of ${\mV}(x^l)$, since both $D_l^{-1}$ and $(\mV(x^l)/D_l)^{-1}$ exist and are  bounded,  it is known that ${\mV}(x^l)^{-1}$ exits  (see \cite{DUN, GU, HA, HHS}, or \cite{Zhangfu} formula 0.8.1 which is similar to the Banachiewicz inversion formula)
\begin{equation}\label{eq:inverse:newton:deri}
{\mV}(x^l)^{-1} = \begin{bmatrix}
(\mV(x^l)/D_l)^{-1} & -(\mV(x^l)/D_l)^{-1}\nabla^*D_l^{-1} \\
-D_l^{-1}C_l 	(\mV(x^l)/D_l)^{-1} & D_l^{-1} + D_l^{-1}C_l	(\mV(x^l)/D_l)^{-1}\nabla^*D_l^{-1}
\end{bmatrix}.
\end{equation}
Together with the boundedness of $C_l$ and $D_l$, we get the boundedness of ${\mV}(x^l)^{-1}$.

Let us turn to the line system \eqref{eq:ssn:dual} for calculating the dual variables $(p^{l+1}, q^{l+1})$ first. Actually, $(D_l -C_l A^{-1}  B )^{-1}$ exists. By the  Sherman--Morrison--Woodbury formula  \cite{HA, HHS} together with the existence of  $D_l^{-1}$ and $(A-BD_l^{-1} C_l)^{-1}$, we have
\begin{equation}\label{eq:inverse:dual:ssn}
(\mV(x^l)/A)^{-1} = (D_l - C_l A^{-1}B)^{-1} = D_l^{-1} + D_l^{-1}C_l (A-BD_l^{-1} C_l)^{-1}BD_l^{-1}.
\end{equation}
The boundedness of $(D_l - C_l A^{-1}B)^{-1}$ follows by the  boundedness of $(A-BD_l^{-1} C_l)^{-1}$ as shown in Theorem \ref{thm:posi:primal} together  with the boundedness of $D_{l}^{-1}$ and $C_l$. 

However, since we found that  solving \eqref{eq:tgv:primaluw:eq} first is much more efficient than solving \eqref{eq:ssn:dual} first according to our numerical experiments, we will only focus on the approach that solving \eqref{eq:tgv:primaluw:eq} first henceforth.
Let us conclude  this section by the following primal-dual semismooth Newton based ALM \eqref{eq:tgv_primal} by giving Algorithm \ref{alm:SSN_PDP} together with projections to the feasible sets of $p$ and $q$ by Remark \ref{rem:pro:feasible}. Although globalization strategies including the Armijo line search are usually needed for the global convergence of Newton methods, however, the \textbf{SSNPDP} in Algorithm \ref{alm:SSN_PDP}  also shares some global convergence numerically, which is also observed  for the corresponding problems in \cite{HS}.
\begin{algorithm}[h]
	\caption{ALM with Primal-dual semismooth Newton with solving primal variables first (\textbf{ALM-PDP})
		\label{alm:SSN_PDP}}
	\begin{algorithmic}
		\STATE {\textbf{ALM}: Given noisy image $f$,  multipliers $\lambda^0$ and $\mu^0$, step size $\sigma_0$ of ALM, iterate the following steps for $k=0, 1,\cdots, $ unless some stopping criterion with the primal-dual form \eqref{eq:tgv-denoising-saddle:original} is satisfied. Do the following \textbf{SSNPDP} for each inner iteration of \textbf{ALM}}: 
		\STATE{\quad \quad \textbf{SSNPDP}: Given initial values $(u^0,w^0)$ and $p^0 \in\{p:\norm[\infty]{p} \leq \alpha_1\}$ and $q^0 \in \{q:\norm[\infty]{q} \leq \alpha_0\}$,  Iterate the following steps: \textbf{Step 1, Step 2, Step 3} for $l=0, 1, \cdots, $ unless some stopping criterion associated with the nonlinear system \eqref{eq:opti:pq} is satisfied. Here $(u^0, w^0, p^0, q^0)$ is usually chosen as $(u^k, w^k, p^k,q^k)$ from the last ALM iteration. }
		\STATE {\quad \quad \textbf{SSNPDP: Step 1: }Solve the linear system \eqref{eq:tgv:primaluw:eq} for $(u^{l+1},w^{l+1})$ with some stopping criterion with iterative method (BiCGSTAB):}
		\STATE{\quad \quad \textbf{SSNPDP: Step 2: }Update $(p^{l+1},q^{l+1})$ by \eqref{eq:update:pd:after:uw}}. 
		\STATE{\quad \quad \textbf{SSNPDP: Step 3: }Project $(p^{l+1},q^{l+1})$ to the feasible set  $\{p:\norm[\infty]{p} \leq \alpha_1\}$ and $\{q:\norm[\infty]{q} \leq \alpha_0\}$, i.e., $p^{l+1}  =\mathcal{P}_{\alpha_1}(p^{l+1})$ and $q^{l+1}  =\mathcal{P}_{\alpha_0}(q^{l+1})$ as the initial values for the next Newton iteration}.
		\STATE{\quad \quad \textbf{SSNPDP: Output} $(u^{k+1}, w^{k+1}, p^{k+1}, q^{k+1})$ by the last $(u^{l+1}, w^{l+1}, p^{l+1}, q^{l+1})$.}
		\STATE{ \textbf{ALM:} Update the Lagrangian multipliers: $\lambda^{k+1} = p^{k+1}$ and
			$\mu^{k+1} = q^{k+1}$ and the step sizes $\sigma_{k+1} = c_0\sigma_{k}$ with $c_0>1$}.
	\end{algorithmic}
\end{algorithm} 

Now, let  us focus on the convergence of  the semismooth Newton solvers \textbf{SSNPDP}  in Algorithm \ref{alm:SSN_PDP}. Although the non-singularity of the corresponding Newton derivative is guaranteed by Theorem \ref{thm:posi:primal} and \eqref{eq:inverse:newton:deri}  through projections to the feasible sets as in  \textbf{SSNPDP}, however, the convergence becomes a subtle issue because the projections to the feasible sets have changed the original Newton derivative. 
Fortunately, similar to Theorem 3.6 of \cite{HS},  we have the following proposition for the convergence of \textbf{SSNPDP}.
\begin{proposition}\label{prop:near}
For the $(k+1)$-th update in Algorithm \ref{alm:SSN_PDP}  with fixed $\lambda^k$, $\mu^k$ and $\sigma_k$, letting $x_{*}^{k+1}:=(u_*^{k+1}, w_*^{k+1}, p_*^{k+1}, q_*^{k+1})$ be the solution of \eqref{eq:opti:pq},  then the iterates $x^l: = (u^l, w^l, p^l, q^l)$ produced by \textbf{SSNPDP}  in Algorithm \ref{alm:SSN_PDP}  converge superlinearly to $x_{*}^{k+1}$ provided that $(u^0, w^0, p^0, q^0)$ is sufficiently close to $x_{*}^{k+1}$.
\end{proposition}
\begin{proof}
    The proof is completely similar to the proof of Theorem 3.6 in \cite{HS}. Here we give a sketch of the proof. Denote $\mathcal{V}_{+}^l$ as the perturbed Newton derivative with the original $p^l$ and $q^l$ in $\mathcal{V}(x^l)$ of \eqref{eq:newton:deri:b}  replaced by $p_{+}^{l}: = \mP_{\alpha_1}(p^l)$ and $q_{+}^{l}: = \mP_{\alpha_0}(q^l)$  correspondingly.  Since the solution $(p_*^{k+1}, q_*^{k+1})$  of \eqref{eq:opti:pq} is feasible and satisfies the constraints, we thus can get the boundedness of $V(x_*^{k+1})^{-1}$ as in \eqref{eq:inverse:newton:deri}. Noting that $\mathcal{V}(x^l)$, especially $C_l$ and $D_l$ are continuously depending on  $(u^l, w^l, p^l, q^l)$ for fixed $\lambda^k$, $\mu^k$ and $\sigma_k$, we conclude that
     for each $\Delta >0$ there exists $\rho>0$ and $x^l$ is in a small $\rho$-ball around $x_*^{k+1}$ as assumed such that 
    \[
    \|\mathcal{V}(x_*^{k+1})- \mathcal{V}_{+}^l \| \leq \Delta.
    \]
    The boundedness $\mathcal{V}_{+}^l$ also follows which means that there exists $C>0$ such that $\| \mathcal{V}_{+}^l   \| \leq C$.
    Now with Theorem 4.1 of \cite{SH}, we conclude that $x^l$ converges to $x_*^{k+1}$ linearly. Furthermore, with this convergence, we can get that $\mathcal{V}_{+}^l$ converges to $\mathcal{V}(x_*^{k+1})$. Finally, with Theorem 4.2 of \cite{SH}, we obtain the superliner convergence of $x^l$ locally.  
\end{proof}
The condition of Proposition \ref{prop:near} can be satisfied if each $x^{k}$ obtained from the previous ALM iteration gives an initial value that is sufficiently close to the solution $x_*^{k+1}$ of \eqref{eq:inverse:newton:deri}. It is known that ALM is essentially the proximal point method applying to the dual problem \cite{Roc1, Roc2}.   
The convergence and the corresponding rate of augmented Lagrangian method are thus closely related to the convergence of the proximal point algorithm. Especially, the local linear convergence of the multipliers or primal and dual variables is mainly determined by the metric subregularities of the corresponding monotone operators \cite{Roc1,Roc2, LE, LU}. Now, let  us turn to the stopping criterion of ALM which is important for its convergence.   With notation $\Phi_k(u,w,h_1,h_2): =L_{\sigma_k}(u,w,h_1,h_2;\lambda^k,\mu^k)$, $h:=(h_1,h_2)$, and  $x_1 = (u,w)$ as before, we follow the stopping criterion for the inexact augmented Lagrangian method which is originated from \cite{Roc1, Roc2} and employed in \cite{LST, ZZST, ZST}
\begin{align}
&\Phi_k(x_1^{k+1}, h^{k+1}) - \inf \Phi_k(x_1,h) \leq \epsilon_k^2/2\sigma_k, \quad \sum_{k=0}^{\infty}\epsilon_k < \infty, \label{stop:a}        \tag{A} \\
& \Phi_k(x_1^{k+1}, h^{k+1}) - \inf \Phi_k(x_1,h) \label{stop:b1}   \leq \frac{\delta_k^2}{2\sigma_k}(\|\lambda^{k+1}-\lambda^k\|^2+ \|\mu^{k+1}-\mu^k\|^2), 
\ \  \sum_{k=0}^{\infty}\delta_k < +\infty, \tag{B1} 
\end{align}
where here and in what follows the distance $x$ from the set $C$ is defined by
\[
\text{dist}(x,C): = \inf\{\|x-x'\|\ | \ x' \in C\}.
\]
\section{Convergence of the Augmented Lagrangian Method}\label{sec:conver:alm}

In this section, we will investigate the global convergence and local convergence rate of the proposed ALM for the problem \eqref{eq:tgv_primal}. We will touch on some necessary tools from convex analysis. Let us introduce some basic definitions and properties of multivalued mappings from convex analysis \cite{DR, LST}. Let $F: X \rightrightarrows  Y $ be a multivalued  mapping. The graph of $F$ is defined as the set
\[
\text{gph} F: = \{ (x,y) \in X\times Y| y\in F(x)\}.
\]
The inverse of $F$, i.e., $F^{-1}:  Y \rightrightarrows  X$ is defined as the multivalued mapping whose graph is $\{(y,x)| (x,y) \in \text{gph} F\}$. Let us introduce the metric subregularity and calmness for  multivalued mappings \cite{DR, LST}, which is important for analyzing the convergence rate and global convergence of ALM.
\begin{definition}[Metric Subregularity \cite{DR}]\label{def:metricregular}
	A mapping $F: X \Longrightarrow Y$ is called metrically subregular at $\bar x$ for $\bar y $ if $(\bar x, \bar y) \in \text{gph} F$ and there exists modulus $\kappa \geq 0$  along with a neighborhoods $U$ of $\bar x$ and $V$ of $\bar y$ such that
	\begin{equation}\label{eq:metricregular}
	\text{dist}(x, F^{-1}(\bar y)) \leq \kappa  \text{dist}(\bar y, F(x) \cap V ) \quad \text{for all} \ \ x \in U.
	\end{equation}
\end{definition}
\begin{definition}[Calmness \cite{DR}]A mapping $S: \mathbb{R}^m \rightrightarrows \mathbb{R}^n$ is called calm at $\bar y$ for $\bar x$ if $(\bar y, \bar x) \in \text{gph} \ S$, and there is a constant $\kappa \geq 0$ along with  neighborhoods $U$ of $\bar x$ and $V $ of  $\bar y$ such that 
	\begin{equation}\label{calmness:def}
	S(y) \cap U \subset S(\bar y) + \kappa |y-\bar y| \mathbb{B}, \quad \forall y \in V.
	\end{equation}
	In \eqref{calmness:def}, $\mathbb{B}$ denotes the closed unit ball in $\mathbb{R}^n$.
\end{definition}
For the relation between the metric subregularity and the calmness, by \cite{DR} (Theorem 3H.3),  $S$ is called calm at $\bar y$ for $\bar x$ if and only if $S^{-1}: \mathbb{R}^n \rightrightarrows \mathbb{R}^m$ is metrically subregular at $\bar x$ for $\bar y$.
Let us now turn to the finite-dimensional space setting in detail. Let us vectorize the images along with other variables for convenience with discrete operators
\begin{align*}
& \nabla_x \in \mathbb{R}^{m\times n}: \mathbb{R}^n \rightarrow \mathbb{R}^m, \ \  \nabla_y\in \mathbb{R}^{m\times n}: \mathbb{R}^n \rightarrow \mathbb{R}^m, \ \ u=(u^1, \cdots, u^n)^{T} \in \mathbb{R}^n, \\
&p=(p_1,\cdots, p_m)^T \in \mathbb{R}^{2m}, \quad p_i = (p_i^1, p_i^2)^T \in \mathbb{R}^2, \\
& \lambda=(\lambda_1,\cdots, \lambda_m)^T \in \mathbb{R}^{2m}, \quad \lambda_i = (\lambda_i^1, \lambda_{i}^2)^T \in \mathbb{R}^2, \\
&q=(q_1,\cdots, q_m)^T \in \mathbb{R}^{3m}, \quad q_i = (q_i^1, q_i^2, q_i^3)^T \in \mathbb{R}^3, \\
& \mu=(\mu_1,\cdots, \mu_m)^T \in \mathbb{R}^{3m}, \quad \mu_i = (\mu_i^1, \mu_i^2, \mu_i^3,)^T \in \mathbb{R}^3,
\end{align*}
%



Let us now turn to the metric subregularity of $\partial \mathfrak{D}$ for the dual problem \eqref{eq:dual:tgv}. Suppose $(\partial \mathfrak{D})^{-1}(0) \neq \emptyset$ and there exists $(\bar \lambda, \bar \mu)$ such that $0 \in \partial \mathfrak{D}(\bar \lambda, \bar \mu)$. Let us introduce
\begin{equation}\label{eq:subgradient:dual}
g(\lambda): = \mI_{\{\norm[\infty]{\lambda} \leq \alpha_1\}}(\lambda),  \ \  \psi(\mu): = \mI_{\{\norm[\infty]{\mu} \leq \alpha_0\}}(\mu).
\end{equation}

The metric subregularity of $\partial \mathfrak{D}$ is very subtle, since the constraint set
\begin{align*}
g(\lambda) = 0 &\Leftrightarrow \left\{\lambda=(\lambda_1, \cdots, \lambda_m)^{T} \ | \ \lambda_i \in \mathbb{R}^{2}, \ |\lambda_i|= \sqrt{ (\lambda_i^1)^2 + (\lambda_i^2)^2} \leq \alpha_1, \ i =1,\cdots, m\right\},\\
\psi(\mu) = 0 &\Leftrightarrow \left\{\mu=(\mu_1,\mu_2, \cdots, \mu_m)^{T} \ | \ \mu_i \in \mathbb{R}^{3}, |\mu_i|= \sqrt{ (\mu_i^1)^2 + (\mu_i^2)^2 + 2 (\mu_i^3)^2} \leq \alpha_0, \ i =1, 2, \cdots, m \right\},
\end{align*}
are not polyhedral sets with $\lambda_i=(\lambda_i^1,\lambda_i^2)^T$ and $\mu_i = (\mu_i^1, \mu_i^2, \mu_i^3)^T$. Introduce
\[
g_i(\lambda_i) = \mI_{\{ |\lambda_i|\leq \alpha_1\}}(\lambda_i), \quad  \psi_i(\mu_i) = \mI_{\{ |\mu_i|\leq \alpha_0\}}(\mu_i),\quad i=1,2,\cdots, m.
\]
Denote $\mI_{\mathbb{B}_{\alpha_1,\Lambda}^i(0)}(x)$ and $\mI_{\mathbb{B}_{\alpha_0,M}^i(0)}(x)$ as the indicator functions for the following $l_2$ ball constraints corresponding to $\lambda_i$ and $\mu_i$, $ i = 1,\cdots, m$ 
	\begin{align}
&	\mathbb{B}_{\alpha_1,\Lambda}^i(0)=\mathbb{B}_{\alpha_1,\Lambda}(0) :=  \left\{ \tilde \lambda:  = (\tilde\lambda^1,\tilde\lambda^2)^{T} \in \mathbb{R}^2 \ | \ |\tilde \lambda| = \sqrt{ (\tilde\lambda^1)^2 + (\tilde\lambda^2)^2} \leq \alpha_1\right\}, \label{eq:l2:close:balls} \\
&	\mathbb{B}_{\alpha_0,M}^i(0)=\mathbb{B}_{\alpha_0,M}(0) :=  \left\{\tilde \mu:= (\tilde \mu^1, \tilde \mu^2,  \tilde \mu^3)^{T} \in \mathbb{R}^3 \ | \ |\tilde \mu| = \sqrt{(\tilde \mu^1)^2 + (\tilde \mu^2)^2+  2 (\tilde \mu^3)^2} \leq \alpha_0\right\}. \notag
	\end{align}
Henceforth, we also use the notations $\mathbb{B}_{a,\Lambda}(\tilde \lambda_0)$ denoting the $l_2$ closed ball with the center $\tilde \lambda_0 \in \mathbb{R}^2$ and radius $a>0$ and $\mathbb{B}_{b,M}(\tilde \mu_0)$ denoting the $l_2$ closed ball with the center $\tilde \mu_0 \in \mathbb{R}^3$ and radius $b>0$ with the same Euclidean distance $|\cdot|$ as in \eqref{eq:l2:close:balls}.

Furthermore, denote $ \mathbb{B}_{a, \Lambda}( \lambda) = \Pi_{i=1}^m \mathbb{B}_{a,\Lambda}(  \lambda_i)$ with $ \lambda = (\lambda_1, \cdots, \lambda_m)^T$ and  $ \mathbb{B}_{a,M}(  \mu) = \Pi_{i=1}^m \mathbb{B}_{a,M}(  \mu_i)$ with $ \mu = ( \mu_1, \cdots, \mu_m)^T$. 
We can thus write
\[
\partial g = \Pi_{i=1}^m\partial g_{i} = \Pi_{i=1}^m\partial  \mI_{\mathbb{B}_{\alpha_1,\Lambda}(0)}(\lambda_i), \quad \partial \psi = \Pi_{i=1}^m\partial \psi_{i} = \Pi_{i=1}^m\partial  \mI_{\mathbb{B}_{\alpha_0,M}(0)}(\mu_i).
\]
It is known that each $\partial  \mI_{\mathbb{B}_{\alpha_1, \Lambda}(0)}(\lambda_i)$ (or $\partial  \mI_{\mathbb{B}_{\alpha_0,M}(0)}(\mu_i)$) is metrically subregular at $(\bar \lambda_i, \bar v_i) \in \text{gph} \partial  \mI_{\mathbb{B}_{\alpha_1,\Lambda}(0)} $ (or $(\bar \mu_i, \bar o_i) \in \text{gph} \partial  \mI_{\mathbb{B}_{\alpha_0,M}(0)}$ ) \cite{YY} (which can also be obtained from \cite{ZM}). For the metric subregularity of $\partial g$, we have the following lemma. 
\begin{lemma}\label{lem:metric:subregular:g}
	For any $(\bar \lambda, \bar v)^T \in \emph{gph} \ \partial g$, $\partial g$ is metrically  subregular at $\bar \lambda$ for $\bar v$.
\end{lemma}
\begin{proof}
	For any $(\bar \lambda, \bar v)^T \in \text{gph} \ \partial g$, and $V$ of a neignborhoods of $\bar \lambda$, since
	\begin{align*}
	&\text{dist}^2(\lambda, (\partial g)^{-1}(\bar v) ) = \sum_{i=1}^m 	\text{dist}^2(\lambda_i, (\partial g_i)^{-1}(\bar v_i) ) \\
	& \leq \sum_{i=1}^m \kappa_i^2 \text{dist}^2(\bar v_i, (\partial g_i)(\bar \lambda_i))	
	\leq \sum_{i=1}^m \max(\kappa_i^2, i=1,\cdots, m) \text{dist}^2(\bar v_i, (\partial g_i)(\bar \lambda_i))	\\
	&  = \max(\kappa_i^2, i=1,\cdots, m) \text{dist}^2(\bar v, (\partial g)(\bar \lambda)).
	\end{align*}
	Thus with choice $\kappa = \sqrt{ \max_{i=1}^m(\kappa_i^2, i=1,\cdots, m) }$, we found that  $\partial g$ is metrically  subregular at $\bar \lambda$ for $\bar v$ with modulus $\kappa$.
\end{proof}
Completely similar, we can obtain metric subregularity of $\partial \psi$. 
\begin{lemma}\label{lem:metric:subregular:j}
	For any $(\bar \mu, \bar o)^T \in \emph{gph} \ \partial \psi$, $\partial \psi$ is metrically  subregular at $\bar \mu$ for $\bar o$.
\end{lemma}

Now we turn to a more general model compared to \eqref{eq:dual:tgv}.
Suppose $h:  \mathbb{R}^{m} \times \mathbb{R}^{2m} \rightarrow  \mathbb{R}$,
\begin{equation}
h(v_1, v_2) := \frac{1}{2}\|v_1\|_{H^{-1}}^2-\frac{1}{2} \|f_0\|_2^2 + \frac{1}{2a}\|v_2\|_{2}^2.
\end{equation}
Introduce  $\mathbb{A} \in \mathbb{R}^{5m\times 3m}: \mathbb{R}^{5m} $ $\rightarrow \mathbb{R}^{3m}$, $\xi := (	\lambda, \mu)^T$, and $\gamma(\xi) :=  g(\lambda)+ \psi(\mu)$. Let us consider the following more general dual problem
\begin{equation}\label{eq:dual:ROF:general}
\max_{\xi \in U\times V} - \mathfrak{D}(\xi):= -\left(h(\mathbb{A}\xi-b) +\gamma(\xi)\right). 
\end{equation}
With $\mathbb{A} := B$ in \eqref{eq:newton:deri:b} and $b = (K^*f_0, 0)^{T}$, we can recover the original dual problem \eqref{eq:dual:tgv} by \eqref{eq:dual:ROF:general}.
Since $g$ and $\psi$ are separated functions on different variables $\lambda$ and $\mu$, by simple calculation together with Lemma \ref{lem:metric:subregular:g} and \ref{lem:metric:subregular:j}, we see $\partial \gamma$ is metrically subregular at $\bar \xi: = (\bar \lambda, \bar \mu)$ for $(\bar v, \bar o)$. However, the metric subregularity of $\partial  \mathfrak{D}$
at $\bar \xi$ is a subtle issue. Fortunately, we can use the Calm intersection theorem \cite{KKU, KKU1} (also see the following Proposition \ref{prop:calm:inter}) to overcome this difficulty. We refer to \cite{AA} for the case of locally strong convex functions.

Noting $\partial \gamma(\xi) = (\partial g(\lambda), \partial \psi(\mu))^T$, let us introduce the following notations for preparations
\begin{equation}\label{eq:eta:y}
\mathbb{A}\xi=\bar y, \quad 
\bar \eta := \mathbb{A}^T (\nabla_{ y} h (y -b))|_{y=\bar y} =(\bar \eta_1, \bar \eta_2, \cdots, \bar \eta_m)^{T},  \quad \bar \eta_i \in \mathbb{R}^3. 
\end{equation}
With the notations in \eqref{eq:eta:y}, let  us introduce the following constraint sets with $p^1$ and $p^2$ to be determined
\begin{align}
&\mathcal{X}: = \{ \xi \ | \ \mathbb{A}\xi=\bar y, \quad -\bar \eta \in \partial \gamma(\xi)\}, \\
&\Gamma_1(p^1) := \{ \xi \ | \ \mathbb{A} \xi - \bar \eta = p^1  \}, \quad 
\Gamma_2(p^2): = \{ \xi \ | \ p_2 \in \bar \eta + \partial \gamma(\xi) \}, \\
&\hat \Gamma(p^1) :=  \Gamma_1(p^1)\cap\Gamma_2(0) = \{ \xi\ | \  p^1 = \mathbb{A}\xi - \bar y, \ 0 \in \bar \eta + \partial \gamma(\xi) \}, \label{eq:hat:gamma}
\end{align}
where $\mathcal{X}$ is actually the solution set of \eqref{eq:dual:ROF:general} by \eqref{eq:eta:y}.
We also need another two set valued mapping,
\begin{align}
&\Gamma(p^1, p^2):=\{\xi \ |\ p^1 = \mathbb{A} \xi -\bar y, \quad p^2 \in \bar \eta + \partial \gamma(\xi) \},\\
&S(p):=\{\xi \  | \ p \in \nabla_{\xi} (h(\mathbb{A}\xi-b))+ \partial \gamma(\xi) \} \Rightarrow \mathcal{X}=S(0).
\end{align}
Actually the metric subregularity of $\partial \mathfrak{D}$ at $(\bar \xi, 0)$ is equivalent to the calmness $S$ at $(0, \bar \xi)$ \cite{DR} (Theorem 3H.3). Now we turn to the calmness of $S$. Furthermore, since our solution set $\mathcal{X}$ is compact by the constraints of $\lambda$ and $\mu$ in \eqref{eq:dual:ROF:general}, by By \cite{ZM} (Proposition 4) (or Proposition 7 in \cite{YY}  for more general cases), we have the following proposition.
\begin{proposition}
	The calmness of $S$ at $(0, \bar \xi)$ is equivalent to the calmness of $\Gamma$ at $(0,0, \bar \xi)$ for any $\bar \xi \in S(0)$.
\end{proposition}
We would use the following calm intersection theorem to prove the calmness of $\Gamma$. 
\begin{proposition}[Calm intersection theorem \cite{KKU, KKU1}]\label{prop:calm:inter}
	Let $T_1: \mathbb{R}^{q_1} \rightrightarrows \mathbb{R}^k$, $T_2: \mathbb{R}^{q_2} \rightrightarrows \mathbb{R}^k$ be two set-valued maps. Define set-valued maps
	\begin{align}
	T(p^1, p^2): &= T_1(p^1)\cap T_2(p^2), \\
	\hat T(p^1):&= T_1(p^1)\cap T_2(0).
	\end{align}
	Let $\tilde x \in T(0,0)$. Suppose that both set-valued maps $T_1$ and $T_2$ are calm at $(0, \tilde x)$ and $T_1^{-1}$ is pseudo-Lipschitiz at $(0,\tilde x)$. Then $T$ is calm at $(0,0, \tilde x)$ if and only if $\hat T$ is calm at $(0, \tilde x)$.
\end{proposition}
Furthermore, we have the following proposition by \cite{Hof} (or \cite{YY}, Lemma 3).
\begin{proposition}\label{pro:gamma1}
	$\Gamma_1$ is calm  at $(0,\bar \xi)$ and  $\Gamma_1^{-1}$ is pseudo-Lipschitiz at $(0,\bar \xi)$.
\end{proposition}
We need the following assumption first, which is actually a mild condition by the optimality conditions in \eqref{eq:optimalites:tgv:pd}. 
\begin{assumption}\label{asump:existence}
	Let us assume that $(\bar  \lambda, \bar \mu) \in \mathcal{X}$ and
	\begin{itemize}
		\item[\emph{i.}] For each $\bar \lambda_i$, either  $\bar \lambda_{i}  \in \emph{int} \mathbb{B}_{\alpha,\Lambda}(0)$ or $\bar \lambda_{i}  \in \emph{bd} \mathbb{B}_{\alpha_1,\Lambda}(0)$ and there exists $\bar g_i \neq 0  $ such that  $\bar g_i   \in \mN_{\mathbb{B}_{\alpha_1, \Lambda}(0)}(\bar \lambda_i)$.
		\item[\emph{ii.}] For each $\bar \mu_i$, either $\bar \mu_{i}  \in \emph{int} \mathbb{B}_{\alpha_0, M}(0)$ or $\bar \mu_{i}  \in \emph{bd} \mathbb{B}_{\alpha_0, M}(0)$ and there exists $\bar \psi_i \neq 0  $ such that  $\bar \psi_i   \in \mN_{\mathbb{B}_{\alpha_0, M}(0)}(\bar \mu_i)$.
	\end{itemize}	
\end{assumption}
With these preparations, inspired by \cite{YY}, we have the following theorem for the metric subregularity of $\partial \mathfrak{D}$.
\begin{theorem}\label{thm:metric:regular:dual:iso}
	For the problem \eqref{eq:dual:ROF:general},  supposing the dual problem has at least one solution $(\bar \lambda, \bar \mu)$ satisfying the Assumption \ref{asump:existence}, then $\partial \mathfrak{D}$ is metrically subregular at $(\bar \lambda,\bar \mu)$ for the origin. 
\end{theorem}
\begin{proof}
	We mainly need to prove the calmness of $\hat\Gamma(p^1)$ in \eqref{eq:hat:gamma} at $(0,\bar \xi)$. Let us first give the outline of the proof. By metric subregularity of $\partial g$, $\partial \psi$ by Lemma \ref{lem:metric:subregular:g}, \ref{lem:metric:subregular:j}, the fact that $\Gamma_1^{-1}$ is pseudo-Lipschitiz, and the calmness of $\Gamma_1$  at $(0,\bar \xi)$ by Proposition \ref{pro:gamma1},  we get calmness of $\Gamma$ at $(0,0, \bar \xi)$ with the Calm intersection theorem in Proposition \ref{prop:calm:inter}. We thus get the calmness of $S$ at $(0, \bar \xi)$ and the metric subregular of $\partial \mathfrak{D}$ at $\bar \xi$ for the origin. Now let  us go to the details and focus on the the calmness of $\hat\Gamma(p^1)$ at $(0,\bar \xi)$.
	Without loss of generality and according to Assumption \ref{asump:existence}, suppose
	\begin{align*}
	& \bar \lambda_{i} \in \text{int}  \mathbb{B}_{\alpha_1,\Lambda}(0), \quad i = 1, \cdots, L_1,\\
	& \bar \lambda_{i} \in \text{bd}  \mathbb{B}_{\alpha_1, \Lambda}(0), \quad  -\bar g_i  \neq 0 \in \mN_{\mathbb{B}_{\alpha_1, \Lambda}(0)}(\bar \lambda_i), \quad i = L_1+1, \cdots, m, \quad 1<L_1<m;\\
	& \bar \mu_{i} \in \text{int}  \mathbb{B}_{\alpha_0, M}(0), \quad i = 1, \cdots, L_2, \\
	& \bar \mu_{i} \in \text{bd}  \mathbb{B}_{\alpha_0, M}(0), \quad  -\bar \psi_i  \neq 0 \in \mN_{\mathbb{B}_{\alpha_0, M}(0)}(\bar \mu_i), \quad i = L_2+1, \cdots, m, \quad 1<L_2<m.
	\end{align*}
	For $i=1,\cdots, L_1$, $\bar \lambda_{i} \in \text{int} \mathbb{B}_{\alpha_1,\Lambda}^i(0)$, we have $-\bar g_i \in \mN_{\mathbb{B}_{\alpha_1, \Lambda}^i(0)}(\bar \lambda_{i}) = \{0\}$. We thus conclude $\bar g_i =  0$ and
	\begin{equation}\label{eq:gamma20:1}
	\Gamma_2^i(0) = \{\lambda \in \mathbb{R}^2 | 0 \in\mN_{\mathbb{B}_{\alpha_1, \Lambda}(0)}(\lambda) \} =  \mathbb{ B}_{\alpha_1,\Lambda}(0), \quad i = 1, \cdots, L_1.
	\end{equation}
	For $i=L_1+1,\cdots, m$, since $\bar \lambda_{i} \in \text{bd}\mathbb{ B}_{\alpha_1, \Lambda}^i(0)$, for any $\lambda \in \mathbb{B}_{\epsilon,\Lambda}(\bar \lambda_{i}) \cap \mathbb{ B}_{\alpha_1,\Lambda}(0)$, we notice either $\lambda \in \text{bd}\mathbb{ B}_{\alpha_1, \Lambda}(0)$ or  $\lambda \in \text{int}\mathbb{ B}_{\alpha_1, \Lambda}(0)$.
	While  $\lambda \in \text{int}\mathbb{ B}_{\alpha_1, \Lambda}^i(0)$, by the definition of $\Gamma_2(0)$, together with $\mN_{\mathbb{B}_{\alpha_1, \Lambda}(0)}(\lambda) = \{0\}$,  we see
	$\bar g_i = 0$, which is contracted with the assumption (i). While  $\lambda \in \text{bd}\mathbb{ B}_{\alpha_1, \Lambda}(0)$, since 
	\[
	\mN_{\mathbb{B}_{\alpha_1, \Lambda}(0)}(\lambda) = \{s \lambda|s\geq 0\}, \quad\mN_{\mathbb{B}_{\alpha_1, \Lambda}(0)}(\bar \lambda) = \{s_1 \bar \lambda|s_1\geq 0\},
	\]
	together with the definition of $\Gamma_2$, we see the only choice is 
	\begin{equation}\label{eq:gamma20:2}
	\Gamma_2^i(0) = \{ \bar  \lambda_{i}\} , \quad i = L_1+1, \cdots, m.
	\end{equation}
	Similarly, for the case of $\mu_i$, by Assumption \ref{asump:existence}, 
	for $i=1,\cdots, L_2$, $\bar \mu_{i} \in \text{int} \mathbb{B}_{\alpha_0,M}(0)$, we have $-\bar \psi_i \in \mN_{\mathbb{B}_{\alpha_1, M}^i(0)}(\bar \mu_{i}) = \{0\}$. We thus conclude $\bar \psi_i =  0$ and
	\begin{equation}\label{eq:gamma20::mu}
	\Gamma_2^{i+m}(0) = \{\mu \in \mathbb{R}^3 | 0 \in\mN_{\mathbb{B}_{\alpha_0, M}(0)}(\mu) \} =  \mathbb{ B}_{\alpha_0,M}(0), \quad i = 1, \cdots, L_2.
	\end{equation}
	Similarly, for $i = L_2+1, \cdots, m$, we have
	\begin{equation}\label{eq:gamma20:mu}
	\Gamma_2^{i+m}(0) = \{ \bar  \mu_{i}\} , \quad i = L_2+1, \cdots, m.
	\end{equation}

	Choose $\epsilon >0$ small enough such that $\mathbb{B}_{4\epsilon, \Lambda}(\bar \lambda_{i}) \subset \mathbb{B}_{\alpha_1,\Lambda}(0)$ for $i=1,\cdots, L_1$ and $\mathbb{B}_{4\epsilon, M}(\bar \mu_{i}) \subset \mathbb{B}_{\alpha_0,M}(0)$ for $i=1,\cdots, L_2$. We thus conclude that
	\begin{subequations}\label{eq:cons:normacone}
		\begin{align}
		\Gamma_2(0) \cap \mathbb{B}_{\epsilon}(\bar \xi) = (&\mathbb{B}_{\epsilon,\Lambda}(\bar \lambda_{1}), \cdots, \mathbb{B}_{\epsilon,\Lambda}(\bar \lambda_{L_1}), \bar  \lambda_{{L_1+1}}, \cdots, \bar  \lambda_{m}, \\
		&\mathbb{B}_{\epsilon,M}(\bar \mu_{1}), \cdots, \mathbb{B}_{\epsilon,M}(\bar \mu_{L_2}), \bar  \mu_{{L_2+1}}, \cdots, \bar  \mu_{m} )^{T},
		\end{align}
	\end{subequations}
	where $ \mathbb{B}_{\epsilon}(\bar \xi): = \Pi_{i=1}^m\mathbb{B}_{\epsilon,\Lambda}(\bar \lambda_{i}) \times \Pi_{i=1}^m\mathbb{B}_{\epsilon,M}(\bar \mu_{i})$.
	Suppose $P=(p_1, \cdots, p_m, q_1,  \cdots, q_m)^{T}$ and $ \xi \in \Gamma_1(P) \cap \Gamma_2(0) \cap \mathbb{B}_{\epsilon}(\bar \xi)$ with  $p_i \in \mathbb{R}^2 $ and $q_i \in \mathbb{R}^3 $, $i=1, 2, \cdots, m$.
	Introduce the following constraint on $\xi = (\lambda, \mu)^T$ 
	\[
	\mathcal{R}: = \{ \xi \ | \ \lambda_{i} \in \mathbb{R}^2, \mu_{i} \in \mathbb{R}^3, \ i=1, \cdots, m \ | \   \lambda_{i} = \bar \lambda_{i}, \ i=L_1+1, \cdots, m; \ \mu_{i} = \bar \mu_{i}, \ i=L_2+1, \cdots, m   \}.
	\]
	We claim that $	\mathcal{R}$ is a convex and closed polyhedral set. It can be seen as follows. For $i=1$ or $i=2$, let  us denote $\bar L_i  = m-L_i$ and  $0_{2L_i\times2m} \in \mathbb{R}^{2L_i\times2m}$, $0_{2 \bar L_i\times2 L_i} \in \mathbb{R}^{2 \bar L_i\times2L_i}$ as the zero matrix whose elements are all zero. Denote $I_{2\bar L_i\times 2\bar L_i} \in \mathbb{R}^{2 \bar L_i\times2 \bar L_i}$ as the identity matrix. Introduce 
	\begin{align}
	&E_{+,\Lambda} = [0_{2L_1\times 2m};
	0_{2 \bar L_1\times2L}    \ I_{2 \bar L\times 2 \bar L_1}] \in \mathbb{R}^{2m \times 2m}, \\
	&E_{-,\Lambda} = [0_{2L_1\times 2m};
	0_{2 \bar L\times2L}    \ -I_{2 \bar L_1\times 2 \bar L_1}] \in \mathbb{R}^{2m \times 2m},\notag  \\
	&E_{\Lambda} = [E_{+,\Lambda}; E_{-,\Lambda}] \in \mathbb{R}^{4m \times 2m}, \quad \bar \lambda_{\mathcal{R}} := [0, \cdots, 0, \bar \lambda_{L_1+1}, \cdots, \bar \lambda_{m}]^T \in \mathbb{R}^{2m}; \\
	&E_{+,M} = [0_{3L_2\times 3m};
	0_{3 \bar L_2\times3L_2}    \ I_{3 \bar L_2\times 3 \bar L_2}] \in \mathbb{R}^{3m \times 3m}, \\
	& E_{-,M} = [0_{3L_2\times 3m};
	0_{3 \bar L_2 \times 3L_2}    \ -I_{3 \bar L_2\times 3 \bar L_2}] \in \mathbb{R}^{3m \times 3m},\notag  \\
	&E_{M} = [E_{+,M}; E_{-,M}] \in \mathbb{R}^{6m \times 3m}, \quad \bar \mu_{\mathcal{R}} := [0, \cdots, 0, \bar \mu_{L_2+1}, \cdots, \bar \mu_{m}]^T \in \mathbb{R}^{3m}.
	\end{align}
	Let us define
	\[
	E_{\Xi} = \begin{bmatrix}
	E_{\Lambda} & 0_{4m\times 3m} \\
	0_{6m\times 2m} & E_{M}
	\end{bmatrix} \in  \mathbb{R}^{10m \times 5m},\quad \bar \xi_{\mathcal{R}} := [\bar \lambda_{\mathcal{R}}, \bar \mu_{\mathcal{R}}]^T \in \mathbb{R}^{5m}. 
	\]
	We conclude $\mathcal{R} = \{ \xi \ | \ E_{\Xi} \xi \leq E_{\Xi} \bar \xi_{\mathcal{R}} \}$ and $\mathcal{R}$  is thus a polyhedral set. Actually, the following set 
	\begin{equation}\label{eq:matrix:equality}
	M(p): =  \{  \xi \ | \    \mathbb{A}  \xi -\bar y= p, \quad  \xi \in \mathcal{R}\} = \{ \xi \ | \ \mathbb{A} \xi -\bar y = p, \ E_{\Xi}  \xi \leq E_{\Xi}  \bar \xi_{\mathcal{R}} \}, \ \ 
	\end{equation}
	is also a polyhedral set. 
	
	Actually, for any $\xi  \in \Gamma_1(p)\cap\Gamma_2(0) \cap \mathbb{B}_{\epsilon}(\bar \xi)=\hat \Gamma(p) \cap \mathbb{B}_{\epsilon}(\bar \xi) $, denote $\tilde \xi$ as its projection on $ M(0)$. Since $\bar \xi \in M(0)$, we thus have
	\[
	\|\xi - \tilde \xi \| \leq \| \xi - \bar \xi\| \leq \epsilon  \Rightarrow \tilde \xi \in \mathbb{B}_{\epsilon}( \xi) \subset \Pi_{i=1}^m\mathbb{B}_{\alpha_1,\Lambda}(\bar \lambda_{i}) \times \Pi_{i=1}^m\mathbb{B}_{\alpha_0,M}(\bar \mu_{i}).
	\]
	Together with $\tilde \xi \in M(0)$ and $\tilde \xi \in \mathcal{R}$, we see $\tilde \xi \in \Gamma_2(0)$  by \eqref{eq:gamma20:1} and \eqref{eq:gamma20:2}. We thus conclude that $\tilde \xi \in \hat \Gamma(0) =\Gamma_1(0)\cap\Gamma_2(0)$.    By the celebrated results of Hoffman error bound \cite{Hof} on the polyhedral set in \eqref{eq:matrix:equality}, for any  $\xi \in\hat \Gamma(p) \cap \mathbb{B}_{\epsilon}(\bar \xi)$, there exists a constant $\kappa$ such that
	\begin{equation}
	\text{dist}(\xi, \hat \Gamma(0)) \leq  \|\xi - \tilde \xi \| = \text{dist}(\xi, M(0)) \leq \kappa \|p\|,\quad \forall \xi  \in \hat \Gamma(p) \cap \mathbb{B}_{\epsilon}(\bar \xi),
	\end{equation}
	since $E_{\Xi}  \xi \leq E_{\Xi} \bar \xi_0$ by $\xi  \in \hat \Gamma(p)= \Gamma_1(p)\cap\Gamma_2(0)$.
	We thus get the calmness of $\hat\Gamma(p)$ at $(0,\bar \xi)$. While $L_1=0$ (or $L_2=0$), i.e., $\bar \lambda_{i} \in \text{bd}  \mathbb{B}_{\alpha_1,\Lambda}(0)$ (or  $\bar \mu_{i} \in \text{bd}  \mathbb{B}_{\alpha_0,M}(0)$), $i=1, \cdots, m$, one can readily check that
	$\hat \Gamma(p) = \Gamma_1(p)\cap\Gamma_2(0) =  \emptyset$ whenever $p_i \neq 0$ with $p=(p_1, \cdots, p_m)^{T}$ (or whenever $q_i \neq 0$ with $q=(q_1, \cdots, q_m)^{T}$). The case $L_1=m$ or $L_2=m$ is similar depending the conditions of $\mathcal{R}$ in \eqref{eq:matrix:equality}. The calmness of $\hat \Gamma$ follows and the proof is finished.
\end{proof}


Henceforth, we denote $\mathcal{X}$ as the solution sets for the dual problem \eqref{eq:dual:tgv}. With the stopping criterion \eqref{stop:a}, \eqref{stop:b1}, and the conditions of the Proposition \ref{prop:near},  we have the following global and local convergence.
\begin{theorem}\label{thm:local:rate}
	For the TGV regularized and perturbed problem \eqref{eq:tgv_primal}, denote the iteration sequence $(u^k, w^k,h_1^k, h_2^k, \lambda^k,\mu^k)$ generated by ALM-PDP with stopping criteria \eqref{stop:a}.
	Then the sequence $(u^k, w^k,h_1^k, h_2^k, \lambda^k,\mu^k)$  is bounded and converges to $(u^*, w^*,h_1^*, h_2^*, \lambda^*,\mu^*)$ which is a saddle point of \eqref{eq:aug:lag:uw1}.  $T_{\mathfrak{D}}:=\partial \mathfrak{D}$ is metrically subregular for the origin under Assumption \ref{asump:existence}. Supposing the  modulus is $\kappa_{\mathfrak{D}}$ and the additional stopping criteria \eqref{stop:b1} is employed, then the sequence $\xi^k=(\lambda^k, \mu^k)$ converges to  $(\lambda^*,\mu^*) \in \mathcal{X}$ and for arbitrary sufficiently large $k$, 
	\begin{equation}\label{eq:convergence:rate:dual}
	\emph{dist}(\xi^{k+1}, \mathcal{X}) \leq \theta_k \emph{dist}(\xi^k, \mathcal{X}),
	\end{equation}
	where
	\[
	\theta_k = [\kappa_{\mathfrak{D}}(\kappa_{\mathfrak{D}}^2 + \sigma_k^2)^{-1/2} + \delta_k](1-\delta_k)^{-1}, \ \emph{as} \ k\rightarrow \infty, \  \theta_k \rightarrow  \theta_{\infty} = \kappa_{\mathfrak{D}}(\kappa_{\mathfrak{D}}^2+ \sigma_{\infty}^2)^{-1/2} < 1.
	\]
\end{theorem}
\begin{proof}
	Since $U\times V$ is finite-dimensional reflexive space and the primal function \eqref{eq:tgv_primal} is l.s.c. proper convex functional and strongly convex, hence coercive. Thus the existence of the solution can be guaranteed \cite{KK} (Theorem 4.25). Furthermore, since $\dom \mathfrak{F} = U\times V$, by Fenchel-Rockafellar theory \cite{KK} (Chapter 4.3) (or Theorem 5.7 of \cite{Cla}), the solution to the dual problem \eqref{eq:dual:tgv} is not empty and 
	\[
	\inf_{u \in U, w \in V} \mathfrak{F}(u,w) = \sup_{ \lambda \in V, \mu \in W} -\mathfrak{D}(\lambda, \mu).
	\]
	By \cite{Roc2} (Theorem 4) (or Theorem 1 of \cite{Roc1} where the augmented Lagrangian method  essentially is equivalent to the proximal point method applying to the dual problem $\partial \mathfrak{D}$), with criterion \eqref{stop:a}, we get the boundedness of $\{\xi^k\}$. The uniqueness of $(u^*,w^*)$ follows from the strongly convexity of $\mathfrak(u,w)$ and the $h_1^*=\nabla u^* -w^*$ and $h_2^*=\mE w^*$ which come from the optimality conditions for $\lambda$ and $\mu$ for $L_{\sigma}(u,w,h_1,h_2;\lambda,\mu)$ in \eqref{eq:aug:lag:uw1}. The boundedness of $(u^k,w^k,h_1^k, h_2^k)$ and convergence of $(u^k,w^k,h_1^k, h_2^k, \lambda^k, \mu^k)$ then follows by \cite{Roc2} (Theorem 4). 
	
	The local convergence rate  \eqref{eq:convergence:rate:dual} with metric subregularity of $T_{\mathfrak{D}}$ by Theorem \ref{thm:metric:regular:dual:iso} and the stopping criteria  \eqref{stop:a}  \eqref{stop:b1} can be obtained from \cite{Roc2} (Theorem 5) (or Theorem 2 of \cite{Roc1}). 
	%
\end{proof}

\begin{remark}
	By strong convexity of $\mathfrak{F}$ on $(u,w)$ of \eqref{eq:tgv_primal}, we get the uniqueness of the primal solution $(u^*,w^*)$. By the optimality conditions \eqref{eq:optimalites:tgv:pd}, we have 
	\[
	[u^*,w^*]^T = -A^{-1}([K^*f_0, 0]^T - B[\lambda^*,\mu^*]^T).
	\]
\end{remark}
\begin{remark}
	If letting $\sigma_{\infty} \rightarrow +\infty$, we can get superlinear convergence rate by Theorem \ref{thm:local:rate}. However, the linear system for Newton updates \eqref{eq:tgv:primaluw:eq} or \eqref{eq:ssn:dual} will be more ill-posed and hard to solve for large $\sigma_k$.  $\sigma_k$ thus can be fixed  without going to $+\infty$ after several iterations, while ALM  can obtain local linear convergence rate \cite{Roc2}. 
\end{remark}
\section{Numerical Experiments}\label{sec:numer}

\subsection{The choice of parameter $a$: tests for PSNR}
The choice of the parameter $a$ in \eqref{eq:tgv_primal} is a subtle issue. The variable $w$ comes from the TGV regularization and does not belong to the original data term $F(u)$ as in \eqref{eq:tgv_primal:o}. However, the strong convexity of $w$ in  \eqref{eq:tgv_primal}  will certainly bring out some advantage for the semismooth Newton solver in \eqref{eq:tgv:primaluw:eq} compared to \eqref{eq:tgv_primal:o}. Surprisingly, adding the strongly convex term $\frac{a}{2}\|w\|_{2}^2$ can experimentally improve the quality of the restored image for many of the corrupted images (e.g., see Figure \ref{fig:psnr:test}). As shown in Table \ref{tab:inflence:a:w}, $a=1$ can bring out better PNSR for many cases including different noise level and different sizes of images presented compared to the case $a=0$ or $a=10^{-8}$, where all cases are computed by first-order primal-dual method \cite{CP}. Henceforth, we choose $a=1$ for our numerical tests. The perturbed TGV regularization \eqref{eq:tgv_primal} can be seen an  modified TGV regularization instead of approximation due to large $a$. We mainly focus on the model \eqref{eq:tgv_primal} with $a=1$ in this paper.
\subsection{Numerical Tests}
For numerical experiments, we focus on the TGV regularized image denoising model for testing all the proposed algorithms, i.e., $K=I$, $\mu=0$, $H=I$, $a=1$, and $f=f_0$. We employ the finite difference discretization of the discrete gradient $\nabla$  and divergence operator $\Div$ \cite{BPK, CP}, which satisfies \eqref{eq:operators:adjoints} and are very convenient for operator actions based implementation. 
Let us introduce the following residuals of $u$, $w$, $\lambda$, and $\mu$ for the primal-dual optimality conditions \eqref{eq:optimalites:tgv:pd} of the saddle-point problem \eqref{eq:tgv-denoising-saddle:original}
\begin{align*}
&\text{res}(u)^{k+1}:  = \|u^{k+1} - f - \Div \lambda^{k+1}\|_F, \quad 
\text{res}(w)^{k+1}:  = \|aw^{k+1} - \lambda^{k+1} - \Div \mu^{k+1}\|_F,  \\
&\text{res}(\lambda)^{k+1} := \|\lambda^{k+1} - \mathcal{P}_{\alpha_1}(\lambda^{k+1}+ c_0 (\nabla u^{k+1}-w^{k+1}))\|_F, \\
& \text{res}(\mu)^{k+1} := \|\mu^{k+1} - \mathcal{P}_{\alpha_0}(\mu^{k+1}+ c_0 \mathcal{E} w^{k+1})\|_F,
\end{align*}
where $\|\cdot\|_F$ denotes the Frobenius norm and $c_0$ is a positive constant and the projections are defined as in \eqref{eq:projections:def}. 
Our stopping criterion and main metric for all the algorithms compared is  the following scaled sum of these residuals,
\begin{equation}\label{eq:stop:cri}
\mathfrak{U}^{k+1}: = (\text{res}(u)^{k+1} + \text{res}(w)^{k+1}+\text{res}(\lambda)^{k+1}+ \text{res}(\mu)^{k+1})/\|f\|_F,
\end{equation} 
which turns out to be very strict as in numerics.
With \eqref{eq:tgv_primal} and \eqref{eq:dual:tgv}, we also introduce the primal-dual gap for comparison (see also \cite{HPRS})
\begin{equation}\label{eq:gap}
\tilde \gap^{k+1}= \tilde \gap^{k+1}(u^{k+1},w^{k+1}, \lambda^{k+1}, \mu^{k+1}) := \mathfrak{F}(u^{k+1}, w^{k+1}) + \mathfrak{D}(\lambda^{k+1}, \mu^{k+1}).
\end{equation}
We will use the following normalized primal-dual gap \cite{CP}
\begin{equation}  \label{eq:l2-tv-gap}
\gap^{k+1} :  = \tilde \gap^{k+1} /NM, \quad \text{with} \ \  NM = N*M, \ \  u^k \in \mathbb{R}^{N\times M}.
\end{equation}

Let us now turn to the stopping criterion for linear iterative solver for Newton updates, i.e.,  BiCGSTAB (biconjugate gradient stabilized method) for each linear system for the Newton update \eqref{eq:tgv:primaluw:eq}  in Algorithm \ref{alm:SSN_PDP}. We use BiCGSTAB (see Figure 9.1 of \cite{VAN}), which is very efficient for nonsymmetric linear system. The following stopping criterion is employed for solving linear systems to get  the Newton updates with BiCGSTAB \cite{HS}, 
\begin{equation}\label{eq:stop:bicg}
\text{tol}_{k+1}: =.1\min\left\{ \left(\frac{\text{res}_k}{\text{res}_0}\right)^{1.5},  \  \frac{\text{res}_k}{\text{res}_0}  \right\},
\end{equation}
which can help catch the superlinear convergence of semismooth Newton. The $\text{res}_k$ in \eqref{eq:stop:bicg} denotes the residual of the corresponding linear system for the Newton update after the $k$-th BiCGSTAB. 

Now, we turn to the most important stopping criterion \eqref{stop:a}, \eqref{stop:b1} of each ALM iteration for determining how many Newton iterations are needed when solving the corresponding nonlinear systems \eqref{eq:opti:pq}. For the criterion \eqref{stop:b1}, a more practical stopping criterion of ALM for cone programming can be found in \cite{CST}.
With $\mathcal{F}$ as in \eqref{eq:opti:pq}, for the $k$-th ALM iteration, we introduce
\begin{align}
&\mathfrak{R}_{k,SSN^l}^l: =  \|u^{l+1}-f + \nabla^*p^{l+1}\|_F + 
\|aw^{l+1} - p^{l+1} + \mE^*q^{l+1}\|_F \label{eq:alm:stop:ssnpt}  \\
&+\|-({\lambda^k}+ {\sigma_k}(\nabla u^{l+1} -w^{l+1}))+\max(1.0, {|{\lambda^k}+ {\sigma_k}(\nabla u^{l+1} -w^{l+1})|}/{\alpha_1}) p^{l+1}\|_F \notag \\
&+ \|- ({\mu^k}+ {\sigma_k}\mE w^{l+1}) + \max(1.0, {|{\mu^k}+ {\sigma_k}\mE w^{l+1}|}/{\alpha_0})q^{l+1} \|_F. \notag
\end{align}

Here $x^{l+1}$ is generated by semismooth Newton iterations in Algorithm \ref{alm:SSN_PDP}  before the projection to the feasible sets of $p$ and $q$. 
We found the following empirical stopping criterion for seimsmooth Newton iterations during each ALM iteration is efficient experimentally, 
\begin{equation}\label{eq:stop:alm}
\mathfrak{R}_{k,SSN^l}^l  \leq {\delta_k}/{\sigma_k},
\end{equation}
where $\delta_k$ is a small parameter which can be chosen as fixed constants including  $10^{-1}$, $10^{-3}$, $10^{-5}$  in our numerical tests. We emphasize that divided by $\sigma_k$ is of critical importance for the convergence of ALM, which is also required by the stopping criterion \eqref{stop:a}, \eqref{stop:b1}.

It can be seen that while $\mathfrak{R}_{k,SSN^l}^l \rightarrow 0$,  $|\mF(x^l)| \rightarrow 0$ and  $x^l = (u^l, w^l, p^l, q^l)$ converges to the solution of \eqref{eq:opti:pq}. We thus can recover $(h_1^l, h_2^l)$ with\eqref{eq:update:multi:pqway} and conclude that
$(u^l, w^l, h_1^l, h_2^l)$ will converges to a minimizer of \eqref{eq:update:primals:alm} by the convexity of $\Phi_k(u,w,h_1,h_2) =L_{\sigma_k}(u,w,h_1,h_2;\lambda^k,\mu^k)$ on $(u, w, h_1, h_2)$. It follows that that $\Phi_k(u^{l}, w^{l}, h_1^{l}, h_2^{l}) - \inf \Phi_k(u,w,h_1,h_2)$ will converge to zero and the stopping criterion \eqref{stop:a} and \eqref{stop:b1} will satisfy eventually, when $\mathfrak{R}_{k,SSN^l}^l \rightarrow 0$.

For numerical comparisons, we mainly choose the accelerated primal-dual algorithm ALG2 (Algorithm 2) in \cite{CP} with an asymptotic convergence rate $\mathcal{O}(1/k^2)$. The implement of ALG2 is based on following  saddle-point formulation of \eqref{eq:tgv-denoising-saddle:original}
\begin{equation}
\mathbb{F}(u,w) + \langle \mathbb{K}(u,w)^T, (p,q)^T \rangle - \mathbb{G}(p,q),
\end{equation}
where $\mathbb{F}(u,w): = F(u) + a/2\|w\|_{2}^2$, $\mathbb{G}(p,q) = \mI_{\{\norm[\infty]{\lambda} \leq \alpha_1\}}(\lambda) + \mI_{\{\norm[\infty]{\mu} \leq \alpha_0\}}(\mu)$, and $\mathbb{K}=B^*$ with $B$ in \eqref{eq:newton:deri:b}. The parameters are as follows \cite{CP}: $\sigma_0$=$\tau_0 = 1/\sqrt{L}$, $L=12$, $\gamma = 0.7*\min(a,1.0)$. 

Here we do not compare with the primal-dual semismooth Newton method as in \cite{HPRS}, which was proposed as a direct solver for the TGV model.  In \cite{HPRS}, additional strong Tikhonov regularizations  on the dual variables $\lambda$ and $\mu$ including  $-\frac{\gamma_1}{2}\|\lambda\|_{2}^2$ and $-\frac{\gamma_2}{2}\|\mu\|_{2}^2$ with fixed $\gamma_1$ and $\gamma_2$ are added to \eqref{eq:frak:L},   which is quite different from the \textbf{SSN-PDP} within the augmented Lagrangian method here.

\begin{figure}
	\begin{center}
		\subfloat[Train]
		{\includegraphics[width=2.0cm]{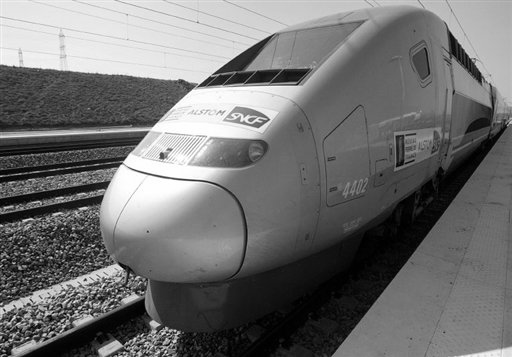}}\ 
		\subfloat[Train 1]
		{\includegraphics[width=2.0cm]{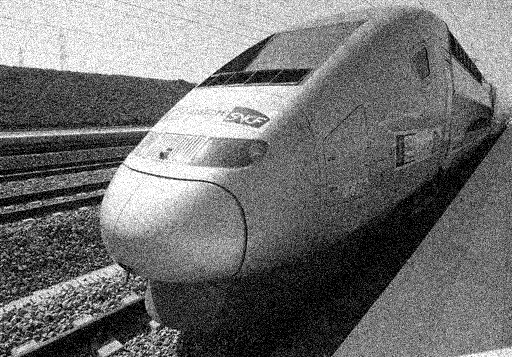}} \ 
		\subfloat[Train 2]
		{\includegraphics[width=2.0cm]{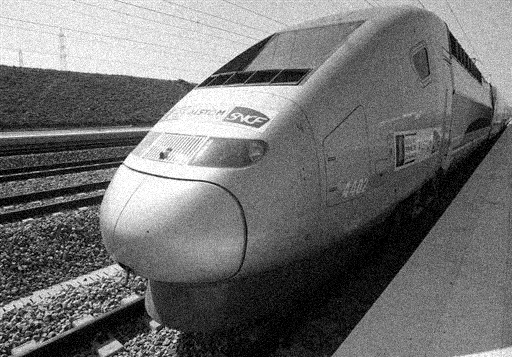}} \ 
		\subfloat[Sails]
		{\includegraphics[width=2.0cm]{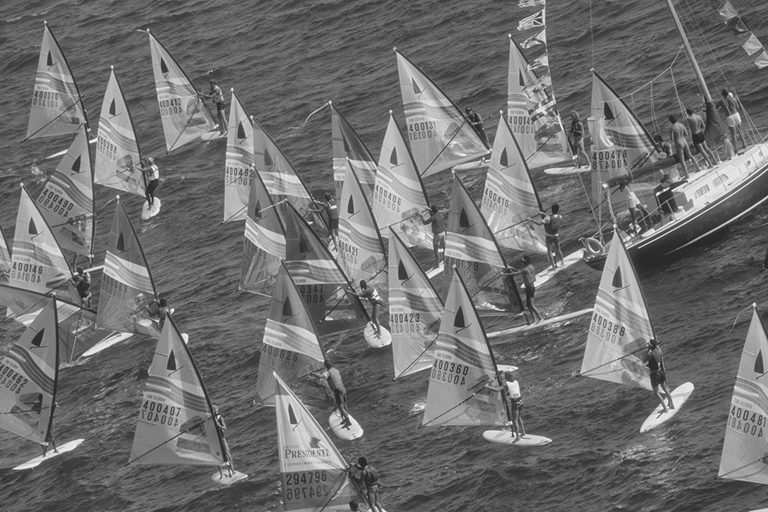}}\ 
		\subfloat[Sails 1]
		{\includegraphics[width=2.0cm]{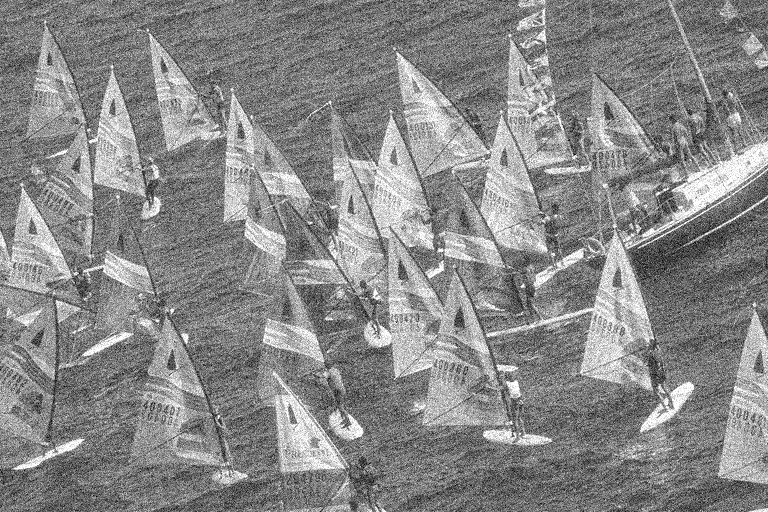}} \ 
		\subfloat[Sails 2]
		{\includegraphics[width=2.0cm]{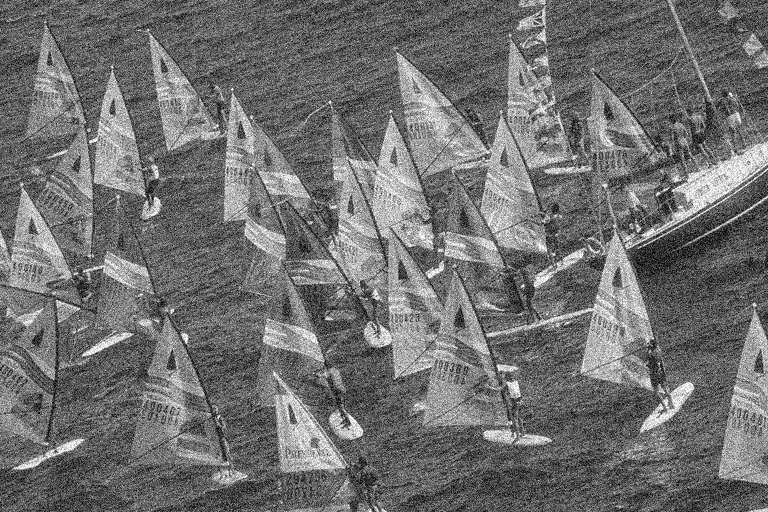}} \\
		\subfloat[Baboon]
		{\includegraphics[width=2.0cm]{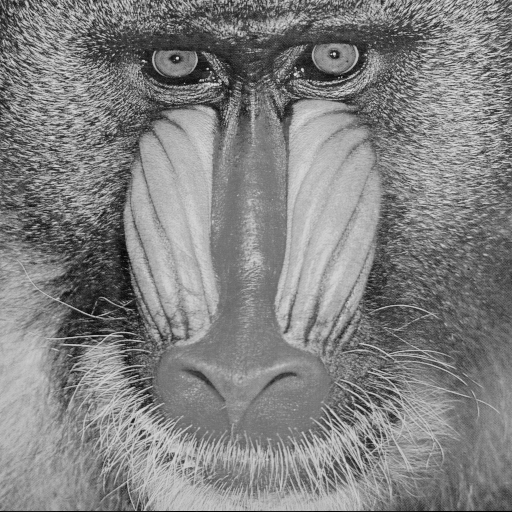}}\ 
		\subfloat[Baboon 1]
		{\includegraphics[width=2.0cm]{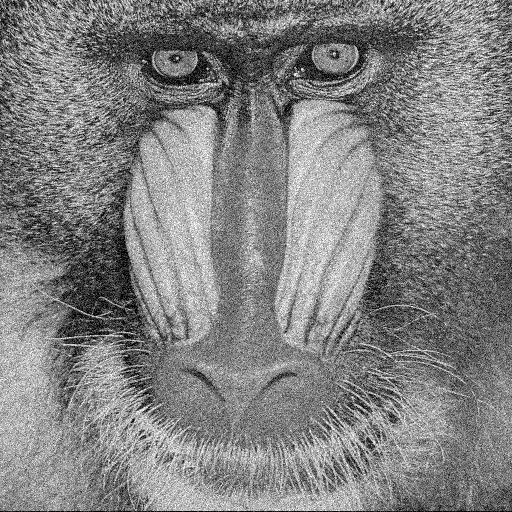}} \ 
		\subfloat[Baboon 2]
		{\includegraphics[width=2.0cm]{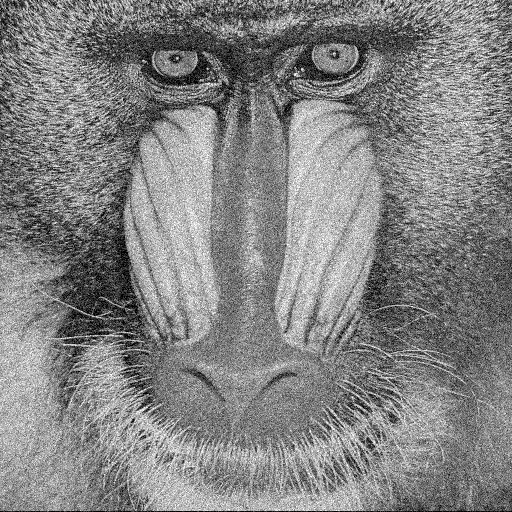}} \
		\subfloat[Man]
		{\includegraphics[width=2.0cm]{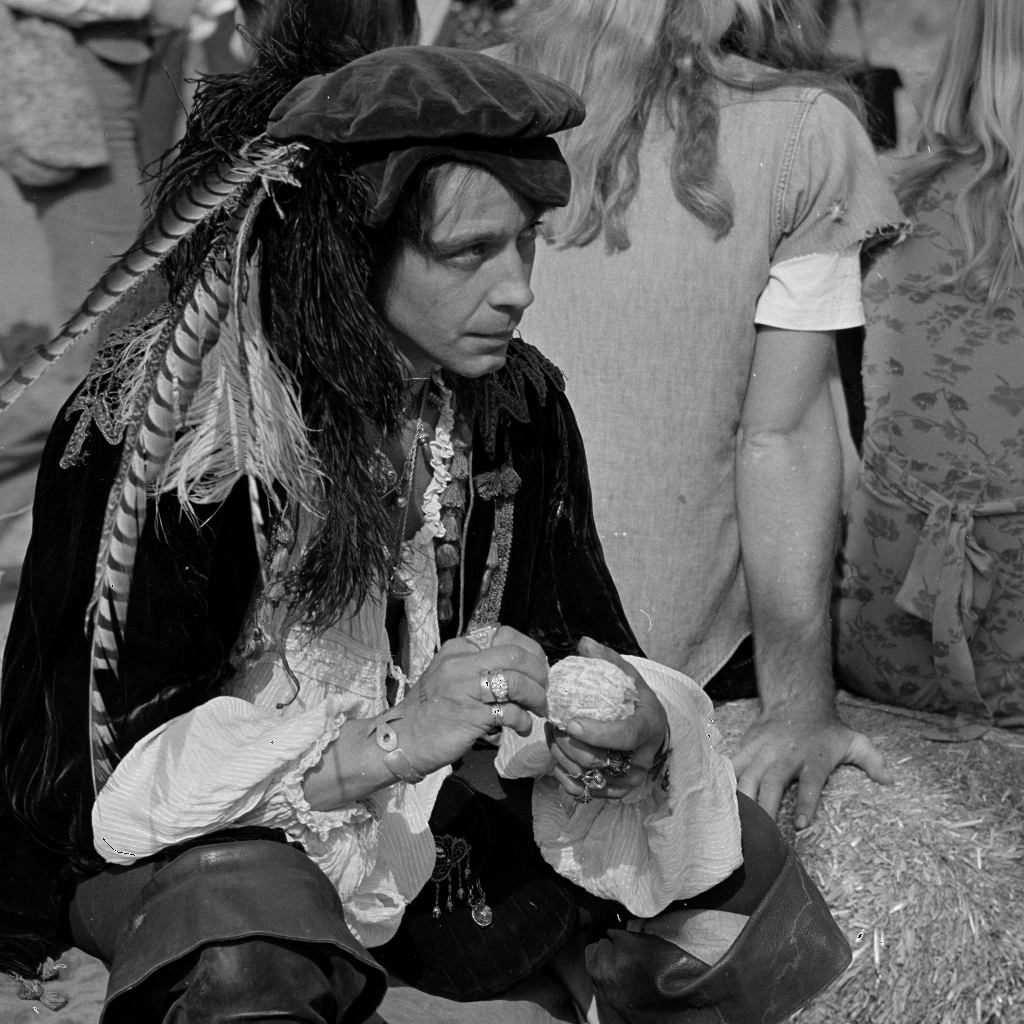}}\
		\subfloat[Man 1]
		{\includegraphics[width=2.0cm]{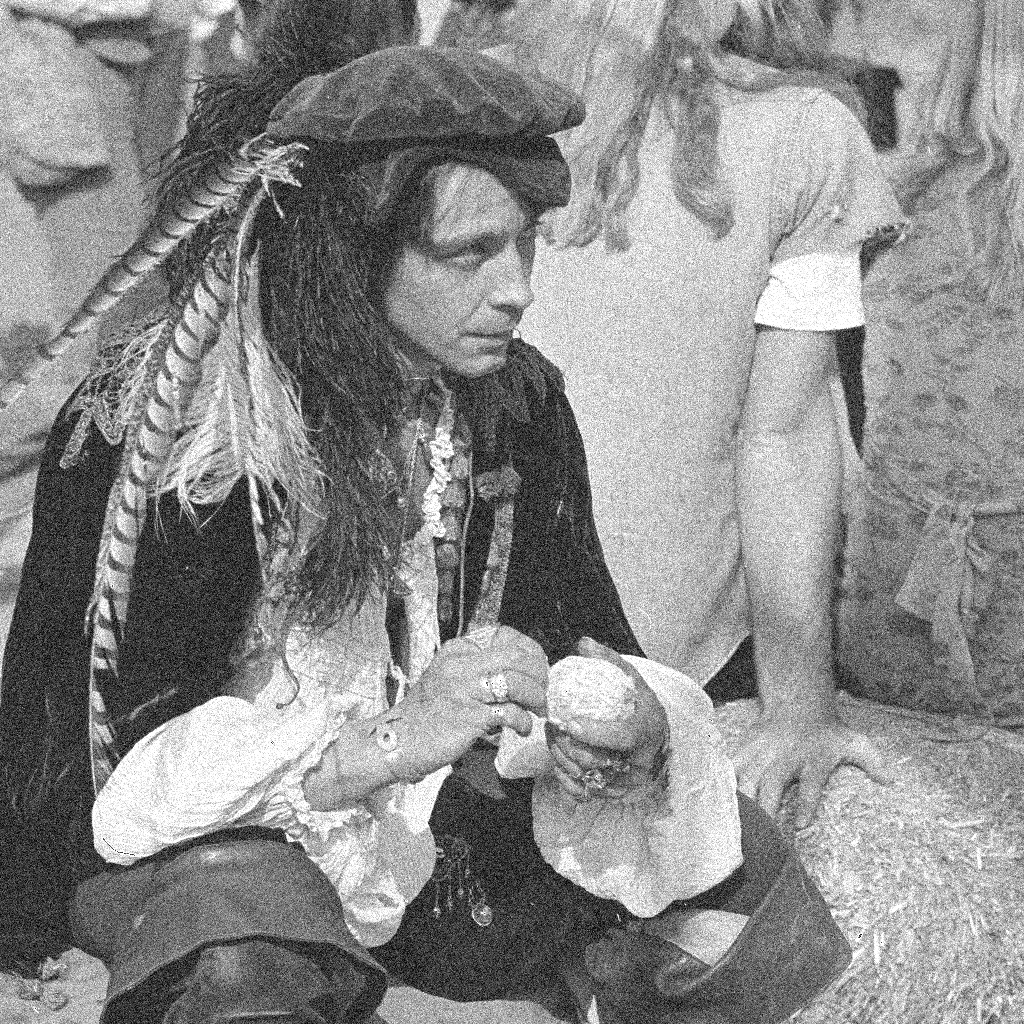}} \ 
		\subfloat[Man 2]
		{\includegraphics[width=2.0cm]{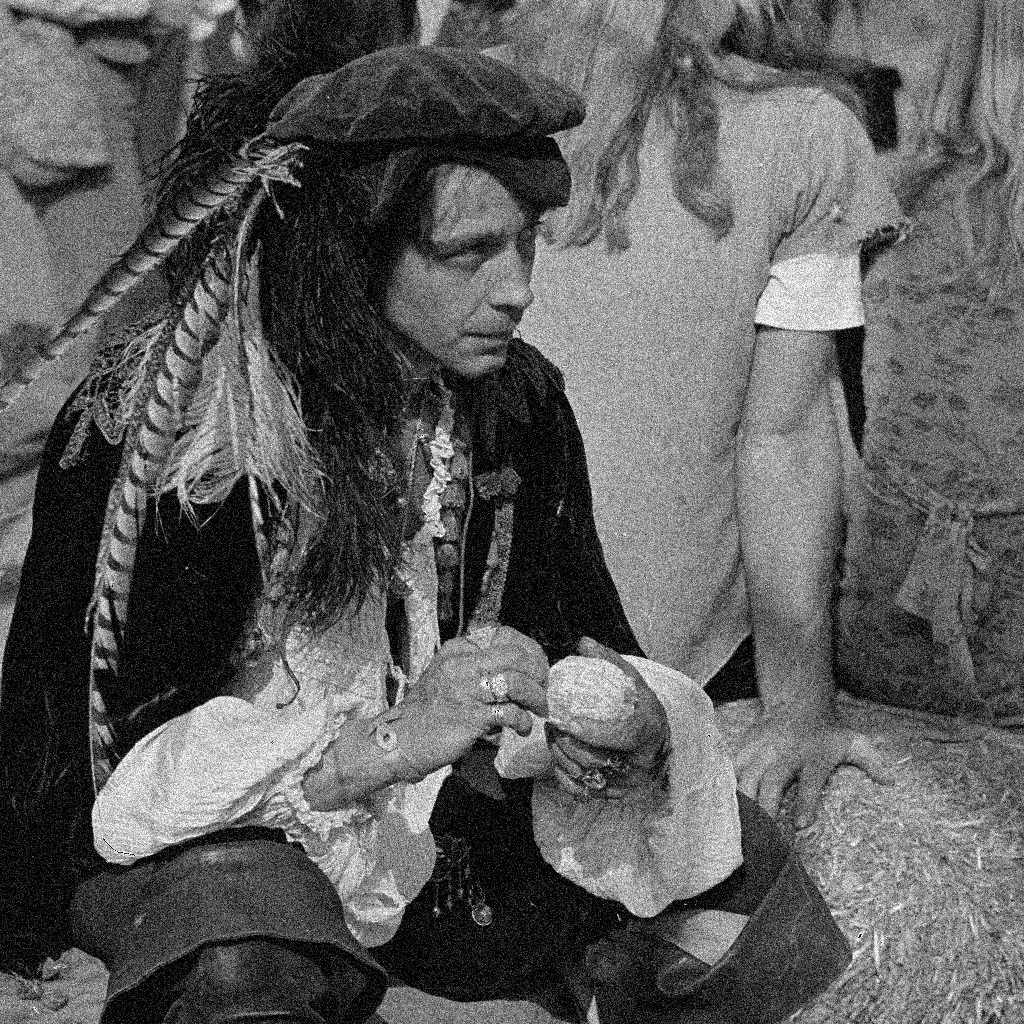}}		
	\end{center}
	\caption{{\small \it Original and corrupted images for PSNR test with TGV regularization. The sizes of the original images are Train: 512 $\times$ 357; Sails: 512 $\times$ 768; Baboon: 512 $\times$ 512; Man: 1024 $\times$ 1024. The corresponding noise level of the corrupted images are Train 1, Baboon 1, Sails 1, Man 2: with 10\% Gaussian noise; Man 1: with 20\% Gaussian noise; Train 2, Baboon 2, Sails 2: with 5\% Gaussian noise. } }
	\label{fig:psnr:test}
\end{figure}

\begin{table}
	\centering 
	\begin{tabular}{rl@{}l@{}l@{}l@{}l@{}l@{}l@{}l@{}l@{}l@{}l@{}l@{}l@{}l@{}l@{}l@{}l@{}l@{}}
		\toprule
		& \multicolumn{6}{c}{$a = 1.0$}
		& \multicolumn{6}{c}{$a = 10^{-8}$}
		& \multicolumn{6}{c}{$a = 0$}\\
		\midrule
		
		& \multicolumn{2}{c}{$\text{PSNR}$}
		& \multicolumn{2}{c}{$\text{RMSE}$}
		& \multicolumn{2}{c}{$\text{SSIM}$}
		& \multicolumn{2}{c}{$\text{PSNR}$}
		& \multicolumn{2}{c}{$\text{RMSE}$}
		& \multicolumn{2}{c}{$\text{SSIM}$}
		& \multicolumn{2}{c}{$\text{PSNR}$}
		& \multicolumn{2}{c}{$\text{RMSE}$}
		& \multicolumn{2}{c}{$\text{SSIM}$}\\
		\midrule
		Train 1&& \textbf{26.604} \ && 2.186e-3 && 7.691e-1 &&  \ 26.595 && 2.190e-3 && 7.688e-1 && \ 26.596  && 2.190e-3 && 7.688e-1\\
		Train 2 && \textbf{29.854} \  && 1.034e-3 && 8.504e-1  && \ 29.835&&1.039e-3&&8.500e-1&& \ 29.835  && 1.039e-3 && 8.500e-1 \\
		Man 1 && \textbf{13.886} \ &&  4.087e-2 && 5.590e-1   && \ 13.886  && 4.087e-2 && 5.589e-1 && \ 13.886  && 4.087e-2 && 5.590e-1\\
		Man 2  && \textbf{27.472} \ && 1.790e-3  && 6.768e-1   && \ 27.469&&1.791e-3&&6.769e-1  && \ 27.470  && 1.790e-3 && 6.769e-1\\
		Baboon 1 && \textbf{22.961} \ && 5.057e-3 && 5.768e-1  && \ 22.952&&5.067e-3&&5.757e-1&&\  22.952  && 5.067e-3 && 5.758e-1 \\
		Baboon 2 && 24.368 \ && 3.658e-3 && 6.990e-1  && \ 24.371&&3.655e-3&&6.993e-1&&\   \textbf{24.371}&&3.655e-3&&6.993e-1\\
		Sails 1 && \textbf{19.027} \ && 1.251e-2 && 6.147e-1  && \ {19.023}&&1.252-2&&6.133e-1  &&\ {19.023}&&1.252-2&&6.133e-1\\
		Sails 2 && {26.168} \ && 2.416e-3 && 6.915e-1  && \ \textbf{26.173}&&2.414-3&&6.918e-1  &&\ {26.172}&&2.414-3&&6.918e-1\\
		\bottomrule 
	\end{tabular}
	\vspace*{-0.5em}
	 \caption{ {\small \it PSNR results of TGV regularized image denoising. Train 1, Baboon 1, Sails 1, Man 2: with 10\% gaussian noise, $\alpha = [0.2, 0.1]$; Man 1: with 20\% gaussian noise, $\alpha = [0.2, 0.1]$; Train 2, Baboon 2, Sails 2: with 5\% gaussian noise, $\alpha = [0.1, 0.05]$.} }
	\label{tab:inflence:a:w} 
\end{table}

\begin{figure}
	\begin{center}
		\subfloat[Original image: Turtle, $128\times128$]
		{\includegraphics[width=4.0cm]{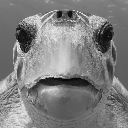}}\quad
		\subfloat[Noisy image: $10\%$ Gaussian]
		{\includegraphics[width=4.0cm]{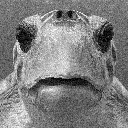}} \quad
		\subfloat[ALM-PDP($10^{-6}$)]
		{\includegraphics[width=4.0cm]{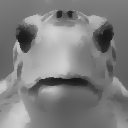}} \\
		\subfloat[Original image: Cameraman, $256\times256$ ]
		{\includegraphics[width=4.0cm]{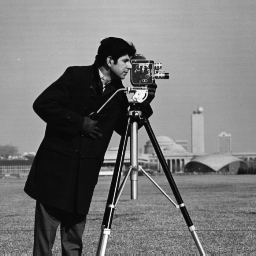}}\quad
		\subfloat[Noisy image: $5\%$ Gaussian]
		{\includegraphics[width=4.0cm]{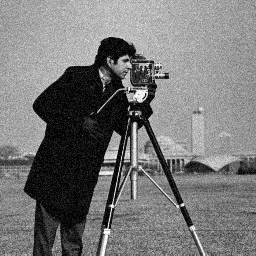}} \quad
		\subfloat[ALM-PDP($10^{-6}$)]
		{\includegraphics[width=4.0cm]{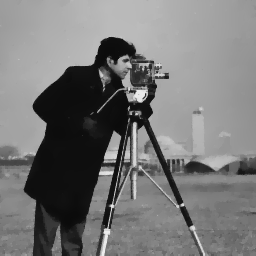}} \\
		\subfloat[Original image: Two macaws, $768\times512$]
		{\includegraphics[width=4.0cm]{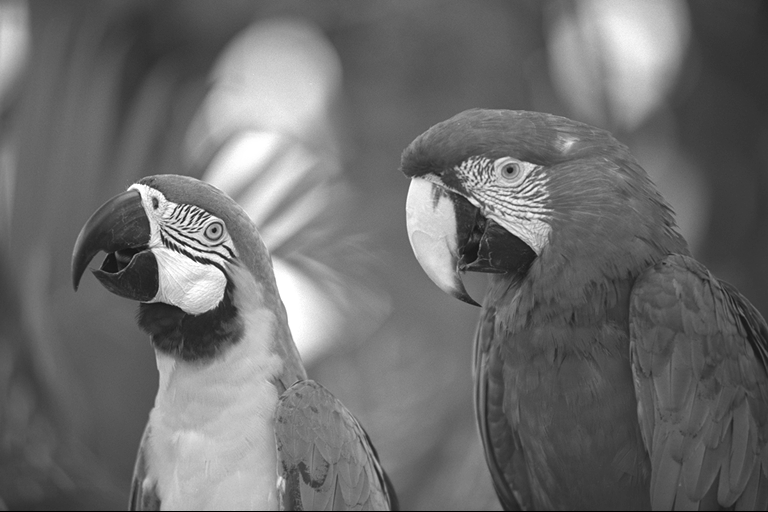}}\quad
		\subfloat[Noisy image: $10\%$ Gaussian]
		{\includegraphics[width=4.0cm]{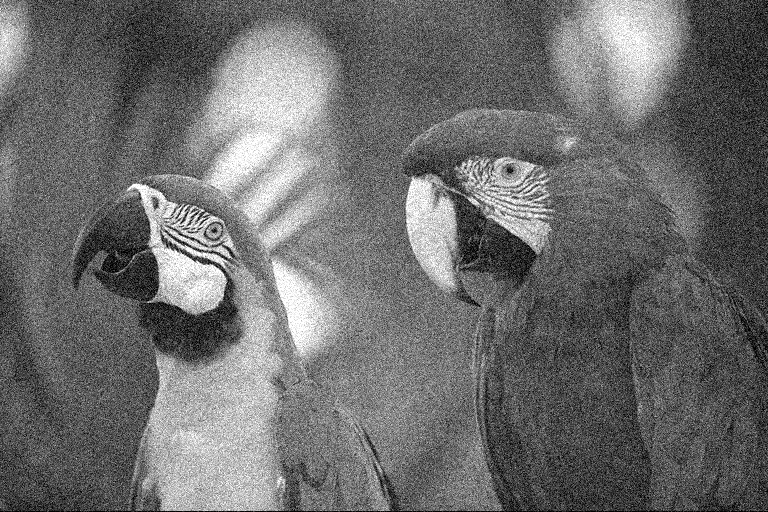}} \quad
		\subfloat[ALM-PDP($10^{-6}$)]
		{\includegraphics[width=4.0cm]{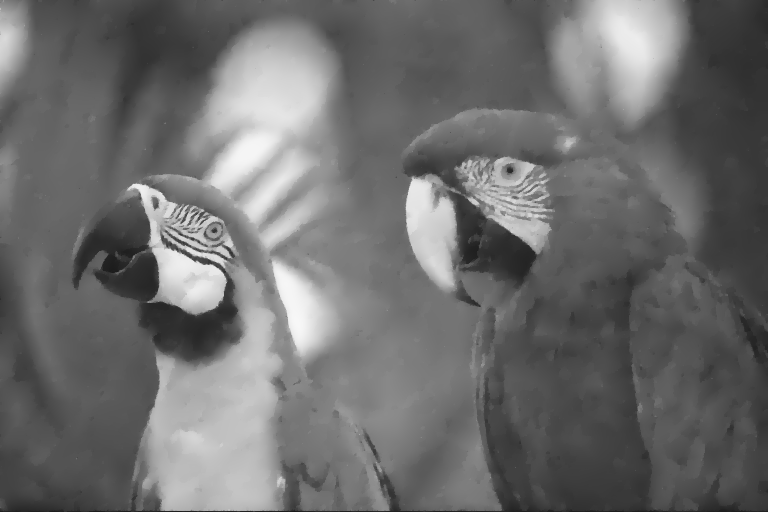}} \\
	\end{center}
	\caption{\small  Images (a), (d), and (g) show the original Turtle, Cameraman and Two macaws images (It is taken from  \url{http://r0k.us/graphics/kodak/kodim23.html} authored by  Kelly S.). (b) and (h) are  noisy versions corrupted by 10\% Gaussian noise; (e) is corrupted by 5\% Gaussian noise. (c), (f), and (i) show the denoised images with ALM-PDP with $\mathfrak{U}^{k+1}<10^{-6}$. Their sizes are: Turle: $128 \times 128$; Cameraman: $256 \times 256$; Two macaws: $768  \times 512$.   }
	\label{fig:main:test}
\end{figure}
\begin{table}
	\centering 
	\begin{tabular}{l*{14}{c}r}
		\midrule
		& \multicolumn{1}{c}{$k=1$}
		& \multicolumn{1}{c}{$k=2$}
		& \multicolumn{1}{c}{$k=3$}
		& \multicolumn{1}{c}{$k=4$}
		& \multicolumn{1}{c}{$k=5$}
		& \multicolumn{1}{c}{$k=6$}
		& \multicolumn{1}{c}{$k=7$}
		\\
		\midrule
		res($u$)       &1.04e-3  &5.67e-4  &2.48e-4  &2.53e-5   &1.09e-6  &4.70e-7  &9.91e-7 \\
		res($w$)       &3.92e-4  &2.47e-4  &1.29e-4  &8.02e-6  &2.14e-7  &1.43e-7  &2.53e-7\\	
		res($\lambda$) &5.55     &1.15     &2.77e-1  &6.89e-2   &1.65e-2  &3.51e-3  &6.20e-4\\
		res($\mu$)     & 1.09e1  &3.02     &8.03e-1  &2.15e-1  &5.19e-2  &1.02e-2  &1.88e-3 \\
		Gap            &4.52e-4  &7.82e-5  &1.83e-5  &4.22e-6   &8.48e-7  &1.42e-7  &2.28e-8   \\
		$N_{SSN}$        &6     &5         &6       &9         &11                  &14   &24   \\
		$N_{ABCG}$       & 10   &13        &27      &43        &75                 &97   &122  \\
		
		\bottomrule 
	\end{tabular}	
	\vspace*{-0.5em}
	\caption{\small \it Image Cameraman denoised by TGV with algorithm ALM-PDP.	$N_{SSN}$ denotes the number of Newton iterations.  $N_{ABCG}$ denotes average number of BiCGSTAB iterations. Here $\sigma_0 = 4$, $\sigma_{k+1} = 4\sigma_k$.}
	\label{tab:cameraman:ssn}
\end{table}


 All computations are done on a laptop with Matlab R2019a.
The details of the test images, the corrupted images,   and the restored image can be found in Figure \ref{fig:main:test}. 

From Tables \ref{tab:turtle}, \ref{tab:cameraman} and \ref{tab:parrot}, it can be seen that the proposed ALM-PDP is very efficient, competitive, and robust for different sizes of images. Especially, the proposed ALM-PDP is highly efficient for high accuracy cases. 
Table \ref{tab:cameraman:ssn} shows the efficiency of the primal-dual semismooth Newton solver. 

We would like to emphasize that our stopping criterion and metric \eqref{eq:stop:cri} is very strict and it directly measures the residuals of the optimality conditions for \eqref{eq:tgv-denoising-saddle:original}. It can be seen from Table \ref{tab:turtle}, \ref{tab:cameraman} and \ref{tab:parrot} that while the gap function \eqref{eq:gap} attains a very low accuracy, the  stopping criterion \eqref{eq:stop:cri} just arrives at a middle-level accuracy.  These unusual observations tell that the proposed ALM is quite appropriate for high-accuracy tasks.

\begin{table}
	\centering 
	\begin{tabular}{l@{}l@{}rl@{}rl@{}rl@{}rl@{}rl@{}rl@{}rl@{}l@{}r} 
		\toprule
		& \multicolumn{6}{c}{TGV: $\alpha = [0.2, 0.1]$ }
		& \multicolumn{8}{c}{Turtle: $128 \times 128$}\\
		\midrule
		& \multicolumn{2}{c}{$n(t)$}
		& \multicolumn{2}{c}{$\text{res}(u)$}
		& \multicolumn{2}{c}{$\text{res}(w)$}
		& \multicolumn{2}{c}{$\text{res}(\lambda)$}
		& \multicolumn{2}{c}{$\text{res}(\mu)$}
		& \multicolumn{2}{c}{Gap}
		& \multicolumn{2}{c}{$\text{PSNR}$}
		& \multicolumn{2}{c}{$\mathfrak{U}$}\\
		\midrule
		ALM-PDP &&7(20.72s)  && 2.40e-3 &&1.42e-3  &&4.35e-4  &&1.80e-3 && 3.60e-8&&24.93&&1e-4\\
		\midrule
		ALG2 &&3176(33.36s)  && 6.38e-3 &&5.77e-4 &&9.37e-6 &&2.20e-5 &&1.53e-9&&24.93&&1e-4\\
		\midrule 
		\midrule 
		ALM-PDP &&9(114.03.s)  && 8.97e-6 &&5.63e-6  &&1.38e-5 &&3.02e-5&&5.93e-10&&24.93&&1e-6\\
		\midrule
		ALG2 &&10808(1033.93s)  && 6.33e-5 &&4.43e-6 &&4.42e-9 &&8.32e-9 &&2.07e-13&&24.93&&1e-6\\
		\bottomrule 
		\bottomrule 
	\end{tabular}
	\vspace*{-0.5em}
	\caption{\small \it For $n(t)$ of the first line of each algorithm, $n$ presents the iteration number for the primal dual gap less than the stopping value; $t$ denotes the CPU time.}
	\label{tab:turtle}
\end{table}

\begin{table}
	\centering 
	\begin{tabular}{l@{}l@{}rl@{}rl@{}rl@{}rl@{}rl@{}rl@{}rl@{}l@{}r} 
		\toprule
		& \multicolumn{6}{c}{TGV: $\alpha = [0.1,0.05]$ }
		& \multicolumn{9}{c}{Cameraman: $256 \times 256$}\\
		\midrule
		& \multicolumn{2}{c}{$n(t)$}
		& \multicolumn{2}{c}{$\text{res}(u)$}
		& \multicolumn{2}{c}{$\text{res}(w)$}
		& \multicolumn{2}{c}{$\text{res}(\lambda)$}
		& \multicolumn{2}{c}{$\text{res}(\mu)$}
		& \multicolumn{2}{c}{Gap}
		& \multicolumn{2}{c}{$\text{PSNR}$}
		& \multicolumn{2}{c}{$\mathfrak{U}$}\\
		\midrule
		ALM-PDP &&7(66.15s)  && 5.71e-3 &&3.84e-3  &&6.49e-4 &&1.90e-3&& 2.32e-8&&30.16&&1e-4\\
		\midrule
		ALG2 &&3115(77.45s)  && 1.27e-2 &&8.31e-4 &&1.02e-5 &&4.65e-5 &&1.62e-9&&30.16&&1e-4\\
		\midrule 
		\midrule 
		ALM-PDP &&9(369.39s)  && 2.25e-5 &&1.27e-5  &&1.49e-5 &&6.64e-5&&5.58e-10&&30.16&&1e-6\\
		\midrule
		ALG2 &&81706(2000.26s)  && 1.26e-4 &&9.11e-6 &&1.60e-8 &&3.06e-8 &&2.85e-13&&30.16&&1e-6\\
		\bottomrule 
		\bottomrule 
	\end{tabular}
	\vspace*{-0.5em}
	\caption{\small \it For $n(t)$ of the first line of each algorithm, $n$ presents the iteration number for the primal-dual gap less than the stopping value, $t$ denoting the CPU time.  }
	\label{tab:cameraman}
\end{table}

\begin{table}
	\centering 
	\begin{tabular}{l@{}l@{}rl@{}rl@{}rl@{}rl@{}rl@{}rl@{}rl@{}l@{}r} 
		\toprule
		& \multicolumn{6}{c}{ TGV: $\alpha = [0.2,0.1]$ }
		& \multicolumn{9}{c}{Two macaws: $768 \times 512$}\\
		\midrule
		& \multicolumn{2}{c}{$n(t)$}
		& \multicolumn{2}{c}{$\text{res}(u)$}
		& \multicolumn{2}{c}{$\text{res}(w)$}
		& \multicolumn{2}{c}{$\text{res}(\lambda)$}
		& \multicolumn{2}{c}{$\text{res}(\mu)$}
		& \multicolumn{2}{c}{Gap}
		& \multicolumn{2}{c}{$\text{PSNR}$}
		& \multicolumn{2}{c}{$\mathfrak{U}$}\\
		\midrule
		ALM-PDP &&3(238.01s)&& 2.43e-2 &&1.39e-2  &&6.55e-2 &&2.28e-5&& 1.70e-6&&31.53&&1e-3\\
		\midrule
		ALG2 &&683(99.84s)   && 2.57e-1 &&3.01e-2 &&3.14e-3 &&7.98e-3 &&1.63e-7&&31.52&&1e-3\\
		\midrule 
		\midrule 
		ALM-PDP &&8(1499.23s)  && 2.79e-4 &&1.02e-4 &&4.52e-4 &&1.66e-3&&9.09e-9&&31.53&&1e-5\\
		\midrule
		ALG2 &&15638(2332.49s)  && 2.82e-3 &&1.55e-4 &&7.40e-7 &&2.48e-6 &&1.66e-11&&31.53&&1e-5\\
		\bottomrule 
		\bottomrule 
	\end{tabular}
	\vspace*{-0.5em}
	\caption{\small \it For $n(t)$ of the first line of each algorithm, $n$ presents the iteration number for the primal-dual gap less than the stopping value, $t$ denoting the CPU time. The notation ``$<$eps" denotes the corresponding quality less than the machine precision in Matlab ``eps".  }
	\label{tab:parrot}
\end{table}
\section{Discussion and Conclusions}\label{sec:conclude}

In this paper, for TGV regularized image restoration, we proposed efficient primal-dual semismooth Newton based ALM algorithms. The corresponding asymptotic local convergence rate along and the global convergence are  discussed by the metric subregularity of the dual functions. Numerical tests show the efficiency of the proposed algorithm.  We would like to emphasize that efficient preconditioners for solving the linear systems involving the Newton updates are very important and desperately needed. Designing efficient preconditioners for the Krylov space based BiCGSTAB especially for large step size $\sigma_k$ is very challenging and useful.

\noindent
{\small
	\textbf{Acknowledgements}
	The author was supported by Beijing Natural Science Foundation No. Z210001.
	The author also acknowledges the support of
	NSF of China under Grant No. \,11701563 and  the support from the program of China Scholarship Council (CSC) under No. 201906365017.	The author is very grateful to Prof. Defeng Sun of Hong Kong Polytechnic University
	for introducing the framework on semismooth Newton based ALM developed by him and his collaborators to the author. The author is also very grateful to Prof. Michael Hinterm{\"u}ller for the discussion on the primal-dual semismooth Newton method during the author's visit to Weierstrass Institute for Applied Analysis and Stochastics (WIAS) supported by Alexander von Humboldt Foundation during 2017. The author is also very grateful to Prof. Kristian Bredies of the University of Graz for the private communications on TGV.
}


\end{document}